\documentclass[preprint,11pt]{elsarticle}

\usepackage{amsfonts, amsmath, amscd}
\usepackage[psamsfonts]{amssymb}

\usepackage{amssymb}

\usepackage{pb-diagram}

\usepackage[all,cmtip]{xy}

\usepackage[usenames]{color}

\newtheorem{theorem}{Theorem}[section]
\newtheorem{lemma}[theorem]{Lemma}
\newtheorem{corollary}[theorem]{Corollary}

\newtheorem{remark}[theorem]{Remark}
\newtheorem{proposition}[theorem]{Proposition}
\newtheorem{definition}[theorem]{Definition}

\newtheorem{fact}[theorem]{Fact}

\newtheorem{problem}[theorem]{Problem}

\newproof{proof}{Proof}

\numberwithin{equation}{section}
\numberwithin{theorem}{section}

\newcommand{\e}{\varepsilon}
\newcommand{\w}{\omega}

\newcommand{\NN}{\mathbb{N}}
\newcommand{\ZZ}{\mathbb{Z}}

\newcommand{\IR}{\mathbb{R}}

\newcommand{\Ss}{\mathbb{S}}

\newcommand{\TT}{\mathbb{T}}

\newcommand{\TTT}{\mathcal{T}}
\newcommand{\FF}{\mathcal{F}}

\newcommand{\V}{\mathcal{V}}
\newcommand{\U}{\mathcal{U}}

\newcommand{\BB}{\mathcal{B}}

\newcommand{\KK}{\mathcal{K}}
\newcommand{\Nn}{\mathcal{N}}
\newcommand{\AAA}{\mathcal A}

\newcommand{\supp}{\mathrm{supp}}

\newcommand{\cl}{\mathrm{cl}}

\newcommand{\Nuc}{\mathsf{Nuc}}

\newcommand{\Mac}{\mathsf{Mac}}

\newcommand{\qc}{\mathsf{qc}}
\newcommand{\Id}{\mathsf{id}}

\newcommand{\LQC}{\mathsf{LQC}}
\newcommand{\MAPA}{\mathsf{MAPA}}

\newcommand{\CC}{C_k}

\newcommand{\SM}{{\setminus}}

\newcommand{\Gsa}{\widehat{G}_{\sigma^\ast}}

\input xy
\xyoption{all}

\begin{document}

\begin{frontmatter}

\title{Compatible group topologies on a locally quasi-convex abelian group \\ and the Mackey group problem}

\author{S.~Gabriyelyan}
\ead{saak@math.bgu.ac.il}
\address{Department of Mathematics, Ben-Gurion University of the Negev, Beer-Sheva, P.O. 653, Israel}

\begin{abstract}
For a locally quasi-convex (lqc) abelian group $G$, we give the first description of all compatible group topologies on $G$ and  apply this result to the Mackey group problem for lqc groups. We characterize lqc abelian groups which are Mackey groups or admit a Mackey group topology and provide a characterization of two Mackey groups whose product is Mackey. We obtain the first characterization of locally convex spaces which are Mackey groups.

\end{abstract}

\begin{keyword}
locally quasi-convex \sep compatible group topology \sep Mackey group \sep Mackey property

\MSC[2010] 22A10  \sep 43A25 \sep 46A03

\end{keyword}

\end{frontmatter}



\section{Introduction}


Let $(E,\tau)$ be a locally convex space. Denote by $E'$ the topological dual space of $E$, and let $\sigma(E',E)$ be the weak${}^\ast$ topology on $E'$. A locally convex vector topology $\nu$ on $E$ is called {\em compatible with $\tau$} if the spaces $(E,\tau)$ and $(E,\nu)$ have the same topological dual space, i.e., $(E,\nu)'=(E,\tau)'$. A family $\mathfrak{S}$ of $\sigma(E',E)$-bounded subsets of $E$ is called {\em saturated} if it is closed under taking subsets, scalar multiples and absolutely convex $\sigma(E',E)$-closed hulls of the union of any two of its members. The following classical result describes the family of all locally convex vector topologies on $E$ compatible with $\tau$.
\begin{theorem}[Mackey--Arens] \label{t:Mackey-Arens}
Let $(E,\tau)$ be a locally convex space. Then:
\begin{enumerate}
\item[{\rm (i)}] A locally convex vector topology $\nu$ on $E$  is compatible with $\tau$ if and only if $\nu$ is the topology of uniform convergence on some saturated family $\mathfrak{S}$ that covers $E$ and consists of $\sigma(E',E)$-compact absolutely convex sets in $E'$ and their subsets.
\item[{\rm (ii)}] There is a finest locally convex vector space topology $\mu$ on $E$ compatible with $\tau$. The topology $\mu$, called the {\em Mackey topology} on  $E$ associated with $\tau$, is the locally convex topology on $E$ of uniform convergence on all $\sigma(E',E)$-compact absolutely convex sets in $E'$.
\end{enumerate}
\end{theorem}
If $\mu=\tau$, the space $E$ is called a {\em Mackey space}. The most important class of Mackey spaces is the class of quasibarrelled spaces. This class is sufficiently rich and contains all metrizable locally convex spaces. In particular, every normed space is a Mackey space.

For an abelian topological group $(G,\tau)$, we denote by $\widehat{G}$ the group of all continuous characters of $(G,\tau)$.
Two topologies  $\mu$ and $\nu$ on  $G$  are said to be {\em compatible } if $\widehat{(G,\mu)}=\widehat{(G,\nu)}$.
We denote by $\mathcal{C}(G,\tau)$ the family of all locally quasi-convex group topologies on $G$  compatible with the original topology $\tau$  (for the definition of locally quasi-convex abelian groups see Section \ref{sec:-polar}). Every locally convex space considered as an abelian topological group is a locally quasi-convex group (all locally convex spaces in the article are assumed to be real). The clause (i) of the Mackey--Arens theorem motivates the following natural problem posed by Mart\'{\i}n Peinador, and which is explicitly stated in the PhD thesis of  de Leo \cite[Problem~8.93]{deLeo}.

\begin{problem} \label{prob:description}
Let $G$ be a locally quasi-convex abelian  group. Describe the poset $\mathcal{C}(G)$. 
\end{problem}
We solve Problem \ref{prob:description} in Theorem \ref{t:compatible-description} which is the main result of the article.

Let  $(G,\tau)$ be a locally quasi-convex (lqc for short) abelian group.
Being motivated by the concept of Mackey spaces the following notion was  introduced and studied in \cite{CMPT}:
The group $(G,\tau)$ is called a {\em Mackey group} if $\nu\leq\tau$ for every $\nu\in\mathcal{C}(G,\tau)$. 
The clause (ii) of the Mackey--Arens theorem suggests the problem of whether every lqc abelian group has a {\em finest} locally quasi-convex group topology. Surprisingly, the answer to this question is negative. Answering a problem posed in \cite{Gab-Mackey}, Au\ss enhofer \cite{Aus3} and the author \cite{Gabr-A(s)-Mackey} independently have proved that the free abelian topological group over a convergent sequence does not admit a Mackey group topology.
In \cite{Gab-Mackey-free}, we proved a much more general assertion which states that the free abelian topological group $A(X)$ over a zero-dimensional metrizable space $X$ has a Mackey group topology if and only if $X$  is discrete.
For numerous results and open problems related to the Mackey problem for lqc abelian groups we refer the reader to \cite{NietoMP} and the very recent survey \cite{AD-mackey}.
Below we recall a characterization of the existence of a Mackey group topology proved in Proposition 3.11 of \cite{CMPT}.
\begin{theorem} \label{t:Char-Mackey-MP}
For an lqc abelian group $(G,\tau)$ the following assertions are equivalent:
\begin{enumerate}
\item[{\rm (i)}] $G$ has a Mackey group topology;
\item[{\rm (ii)}] $\tau_1\vee\tau_2$ is compatible with $\tau$ for every locally quasi-convex group topologies $\tau_1$ and $\tau_2$ on $G$ compatible with $\tau$.
\end{enumerate}
\end{theorem}
In Theorem \ref{t:Mackey-exist} we give a more precise characterization of lqc abelian groups admitting a Mackey group topology, and in Theorem \ref{t:Mackey-group-equi} we obtain a characterization of  Mackey groups.

It is well known (see Theorem 8.8.5 in \cite{Jar}) that the product of an arbitrary family of Mackey locally convex spaces is a Mackey space. One can naturally ask whether an analogous result holds true in the realm of lqc abelian groups. The following problem is Problem 8.82 of \cite{deLeo}:

\begin{problem} \label{prob:product-Mackey}
Let $\{G_i\}_{i\in I}$ be a family of Mackey groups. Determine under which conditions the product $\prod_{i\in I} G_i$ is a Mackey group.
\end{problem}
In Theorem \ref{t:Mackey-product} we answer Problem \ref{prob:product-Mackey} for the case of two groups.

Our proofs of the aforementioned Theorems \ref{t:compatible-description}, \ref{t:Mackey-exist}, \ref{t:Mackey-group-equi} and \ref{t:Mackey-product} are based on the method of constructing  of {\em polar} group topologies coming from Functional Analysis (this method is described in detail in Lemma \ref{l:topology-G-dual} below). 

Let $(E,\tau)$ be a locally convex space (lcs for short). The next question posed in \cite{Gab-Mackey} arises naturally.
\begin{problem} \label{prob:lcs-Mackey-group}
Under which conditions a Mackey lcs $E$ is also a Mackey group?
\end{problem}
Let us note that the condition of being a Mackey space in Problem \ref{prob:lcs-Mackey-group} is necessary. Indeed, Lemma 2.3 of \cite{Gabr-L(X)-Mackey} states that if an lcs $(E,\tau)$ is a Mackey group, then it is a Mackey space. Also one can ask whether the converse to that lemma is true, that is: {\em Is every Mackey space also a Mackey group}? It turns out that the answer to this  question is negative because there are metrizable lcs which are not Mackey groups, see \cite{Gab-Mackey}.   Moreover, very recently we show in \cite{Gabr-normed-Mackey} that there are even normed spaces which are not Mackey groups. 
In the main result of Section \ref{sec:Mackey-group-lcs},  Theorem \ref{t:Mackey-space-group}, we give the first characterization of locally convex spaces which are Mackey groups answering Problem \ref{prob:lcs-Mackey-group}.



\section{Notations and preliminaries} \label{sec:-polar}


In what follows we need some notations. Set $\NN:=\{1,2,\dots\}$ and $\w:=\{0,1,\dots\}$. Denote by $\Ss$ the unit circle group and set $\Ss_+ :=\{z\in  \Ss:\ {\rm Re}(z)\geq 0\}$.  If $z\in \Ss$, then $\arg(z)\in (-\pi,\pi]$.
For every $n\in\NN$, set
\[
\left(\tfrac{1}{n}\right)\Ss_+ =\Ss_n :=\big\{z\in \Ss: \arg(z)\in \left[-\tfrac{\pi}{2n},\tfrac{\pi}{2n}\right]\big\}.
\]

Let $\AAA$ and $\BB$ be to families of subsets of a set $\Omega$. Then $\BB$ {\em swallows} $\AAA$ if for every $A\in\AAA$ there is $B\in\BB$ such that $A\subseteq B$, and $\AAA$ is called {\em directed} if $\AAA$ swallows $\AAA\cup \AAA$ (i.e., for every $A_1,A_2\in\AAA$ there is $A\in\AAA$ such that $A_1 \cup A_2\subseteq A$).

Let $G$ be an abelian group. The order of an element $g\in G$ is denoted by $o(g)$. The number $\exp(G):=\sup\{ o(g):g\in G\}$ is called the {\em exponent} of $G$. If $\exp(G)$ is finite we say that the group $G$ is of {\em finite exponent} or {\em bounded}, and if $\exp(G)=\infty$ the group $G$ is of {\em infinite exponent} or {\em unbounded}. If $\kappa$ is a cardinal number, we denote by $G^{(\kappa)}=\bigoplus_{i\in\kappa} G$ the direct sum of $\kappa$-many of $G$, and the  set $\supp(g):=\{i\in\kappa: g_i\not=0\}$ is called the {\em support} of  $g=(g_i)\in G^{(\kappa)}$.
If $A$ is a subset of $G$ and $n\in\NN$, we define $(1)A:=A$, $(n+1)A:=(n)A +A$,
\[
A_{(n)}:=\{ g\in G: g, 2g,\dots, ng\in A\} \;\; \mbox{ and } \;\; nA:=\{na: a\in A\},
\]
and denote by $\langle A \rangle$ the subgroup of $G$ generated by $A$.


Let $(G,\tau)$ be an abelian topological group. Denote by $\Nn(G,\tau)$ ($\Nn_s(G,\tau)$) the family of all (resp. symmetric)  neighborhoods at the identity of $G$. If $U$ is a closed (open) neighborhood of zero and $n\in\NN$, then the set $U_{(n)}$ is a closed (open) neighborhood of zero (indeed, if $g\in \cl(U_{(n)})$ and a net $\{g_i\}_{i\in I}\subseteq U_{(n)}$ converges to $g$, then for every $k=1,\dots,n$, we have $kg=\lim_i kg_i$ belongs to $U$ by the closeness of $U$ and hence $g\in U_{(n)}$).  The group $G$ endowed with the discrete topology is denoted by $G_d$.
The {\em supremum} $\sup\{\tau_i:i\in I\}$, denoted usually by  $\vee\{\tau_i:i\in I\}$, of a family $\{\tau_i:i\in I\}$ of group topologies on $G$ is the group topology on $G$ whose base of neighborhoods at zero is formed by the sets
\[
U_{i_1}\cap\cdots\cap U_{i_n},\; \mbox{ where } n\in\NN \mbox{ and } U_{i_k}\in\Nn(G,\tau_{i_k}) \mbox{ for } k=1,\dots,n.
\]
We denote by $G^\ast$ the group of all characters (=homomorphisms from $G$ into $\Ss$) of $G$, and let $\widehat{G}$ be the group of all continuous characters of $G$. Note that $G^\ast =\widehat{G_d}$. If $\widehat{G}$ separates the points of $G$, the group $G$ is called {\em maximally almost periodic} ($MAP$ for short). We denote by $\MAPA$ the class of all $MAP$ abelian groups. Let $H$ and $S$ be subgroups of $G$ and $\widehat{G}$, respectively.
The weakest group topology on $G$ under which the elements of $S$ are continuous is denoted by $\sigma(G,S)$. If $S=\widehat{G}$, the topology $\sigma(G,\widehat{G})$ is called the {\em weak topology} on $G$. It is clear that $\sigma(G,\widehat{G})\leq \tau$, and $\sigma(G,\widehat{G})$ is Hausdorff if and only if $G$ is $MAP$. Analogously, the weakest group topology on $\widehat{G}$ under which the elements of $H$ are continuous is denoted by $\sigma(\widehat{G},H)$. In the case  $H=G$, the topology $\sigma(\widehat{G},G)$ is called the {\em weak${}^\ast$ topology} on $\widehat{G}$. For simplicity of notations we set $\Gsa:=\big(\widehat{G},\sigma(\widehat{G},G)\big)$.

Let $(G,\tau)$ be an abelian topological group. If $\chi\in G^\ast$, we will use  also the notation $(\chi,g)$ to denote $\chi(g)$ for $g\in G$. A subset $A$ of $G$ is called {\em quasi-convex} if for every $g\in G\SM A$ there exists   $\chi\in \widehat{G}$ such that $\chi(g)\notin \Ss_+$ and $\chi(A)\subseteq \Ss_+$.  For example, by Example 6.4 of \cite{Aus}, the set $\Ss_n$ is quasi-convex for every $n\in\NN$. An abelian topological group is called {\em locally quasi-convex} ({\em lqc} for short) if it admits a neighborhood base at zero consisting of quasi-convex sets.
The family of all quasi-convex neighborhoods at $0\in G$ is denoted by $\Nn_{qc}(G)$. We denote by $\LQC$ the family of all lqc abelian groups.
For every sets $A\subseteq G$ and  $B\subseteq \widehat{G}$, we define
\[
\begin{aligned}
A^\triangleright & :=\{\chi\in \widehat{G}:\ \chi(A)\subseteq \Ss_+\} \quad (\mbox{the polar of $A$ in $\widehat{G}$}) \\
A^{\blacktriangleright} & :=\{\chi\in G^\ast:\ \chi(A)\subseteq \Ss_+\}\quad (\mbox{the polar of $A$ in $G^\ast=\widehat{G_d}$}) \\
A^\perp & :=\{ \chi\in\widehat{G}: (\chi,g)=1, \mbox{ for all } g\in A\} \quad (\mbox{the annihilator of $A$ in $\widehat{G}$})\\
B^{\triangleleft} & :=\{ g\in G: (\chi,g)\in \Ss_+ \mbox{ for all } \chi\in B\} \quad (\mbox{the inverse polar of $B$ in $G$})\\
B^{\top} & :=\{ g\in G: (\chi,g)=1 \mbox{ for all } \chi\in B\} \quad (\mbox{the inverse annihilator of $B$ in $G$})
\end{aligned}
\]
and the sets $\qc(A)=\qc_G(A):=A^{\triangleright\triangleleft}$ and $\qc_{\widehat{G}}(B):=B^{\triangleleft\triangleright}$ are called the {\em quasi-convex hulls} of $A$ and $B$, respectively. Note that $A^\triangleright$ and $B^{\triangleleft}$ are closed and quasi-convex.
We shall use repeatedly the next simple lemma in which (i)-(viii) are standard or folklore (we give its detailed proof for the reader convenience, to make the article self-contained and for  further references).
\begin{lemma} \label{l:polar-1}
Let $(G,\tau)$ be an abelian topological group, and let $A, K$ and $F$ be subsets of $\widehat{G}$ containing zero. Then:
\begin{enumerate}
\item[{\rm(i)}] if $n\in\NN$, then $g\in \big((n)A \big)^{\triangleleft}$ if and only if $(\chi,g)\in \left(\tfrac{1}{n}\right)\Ss_+=\Ss_n$ for every $\chi\in A$;
\item[{\rm(ii)}] $\big( K+K \big)^{\triangleleft} +\big( K+K \big)^{\triangleleft}\subseteq  K^{\triangleleft}$;
\item[{\rm(iii)}] $\big( K+K \big)^{\triangleleft} \cap \big( F+F \big)^{\triangleleft} \subseteq \big( K +F\big)^{\triangleleft} \subseteq K^{\triangleleft} \cap F^{\triangleleft}$;
\item[{\rm(iv)}] if $n\in\NN$ and if $V\subseteq G$ is such that $(n)V\subseteq K^{\triangleleft}$, then $V \subseteq \big((n)K \big)^{\triangleleft}$;
\item[{\rm(v)}] if $K$ and $F$ are subgroups of $\widehat{G}$, then $\big( K +F\big)^{\triangleleft} = K^{\triangleleft} \cap F^{\triangleleft}$;
\item[{\rm(vi)}] if $n\in\NN$ and $C\subseteq G$ contains zero, then
\[
\begin{aligned}
& \big((n)\qc(C)\big)^{\triangleright}=\big( (n)C \big)^{\triangleright}= (C^{\triangleright})_{(n)},\quad  \qc\big((n)C\big)+\qc\big((n)C\big)\subseteq \qc\big((2n)C\big)\\
& \big((n)\qc_{\widehat{G}}(A)\big)^{\triangleleft}=\big( (n)A \big)^{\triangleleft};
\end{aligned}
\]
consequently, the set $X:=\bigcup_{n\in\NN} \qc\big((n)C\big)$ is a subgroup of $G$;
\item[{\rm(vii)}] $(A^{\triangleleft})_{(n)} = \big( (n)A\big)^{\triangleleft}$ for every $n\in\NN$, so $\bigcap_{n\in\NN} (A^{\triangleleft})_{(n)}$ is a subgroup of $G$;
\item[{\rm(viii)}] if $n,k\in\NN$ and $C\subseteq G$ is quasi-convex, then $C_{(n)}$ is quasi-convex and $C_{(2n)}+C_{(2n)}\subseteq C_{(n)}$; so $(2^k)C_{(2^k n)}\subseteq C_{(n)}$ and  $\bigcap_{n\in\NN} C_{(n)}$ is a subgroup of $G$;
\item[{\rm(ix)}]  if $n\in\NN$,  $H$ is a subgroup of $G$ and $C\subseteq G$, then $(C\cap H)_{(n)}=C_{(n)}\cap H$;
\item[{\rm(x)}] if $H$ is a subgroup of $G$ and $D=(H\cap A^{\triangleleft})^{\triangleright_H}$ $\big($where the second polar $\triangleright_H$ is taken in the dual group $\widehat{H}$ of $(H,\tau{\restriction}_H)$$\big)$, then
    \[
    (D^{\triangleleft_H})_{(n)}=\big( (n)D\big)^{\triangleleft_H}=H\cap \big( (n)A\big)^{\triangleleft}=H\cap(A^{\triangleleft})_{(n)} \;\; \mbox{ for all }\; n\in\NN;
    \]
\item[{\rm(xi)}]  if $C\subseteq G$ contains zero, then $(C+C)^\triangleright +(C+C)^\triangleright\subseteq C^\triangleright$;
\item[{\rm(xii)}]  if $C\subseteq G$  contains zero, then $\big((n)C\big)^\triangleright =(C^\triangleright)_{(n)}$ for every $n\in\NN$.
\end{enumerate}
\end{lemma}

\begin{proof}
(i) Let $g\in \big((n)A \big)^{\triangleleft}$. If $\chi\in A$, then the inclusion $0\in A$ implies $(\chi,g),\dots, (n\chi,g)=(\chi,g)^n\in \Ss_+$. By induction on $n$, it is easy to see that
$ 
(\chi,g)\in \left(\tfrac{1}{n}\right)\Ss_+.
$ 
Conversely,  if $(\chi,g)\in \left(\tfrac{1}{n}\right)\Ss_+$ for every $\chi\in A$, then
\[
(\chi_1+\dots+\chi_n, g)=(\chi_1,g)\cdots(\chi_n,g)\in \left(\tfrac{1}{n}\right)\Ss_+ \stackrel{n}{\cdots} \left(\tfrac{1}{n}\right)\Ss_+ =\Ss_+
\]
for every $\chi_1,\dots,\chi_n\in A$. Thus $g\in \big((n)A \big)^{\triangleleft}$.

(ii) Let $\chi\in K$, and let $g,h\in \big(K+K \big)^{\triangleleft}$. Then, by (i) applied to $n=2$, we obtain
\[
(\chi,g+h)=(\chi,g)\cdot (\chi,h)\in \left(\tfrac{1}{2}\right)\Ss_+ \cdot \left(\tfrac{1}{2}\right)\Ss_+ =\Ss_+.
\]
Thus $g+h\in K^{\triangleleft}$.

(iii) Let $g\in \big( K+K \big)^{\triangleleft} \cap \big( F+F \big)^{\triangleleft}$, and let $\chi\in K$ and $f\in F$. Since $0\in K\cap F$, (i) implies
\[
(\chi+f,g)=(\chi,g)\cdot (f,g) \in \left(\tfrac{1}{2}\right)\Ss_+ \cdot \left(\tfrac{1}{2}\right)\Ss_+ =\Ss_+.
\]
Thus $g\in (K+F)^{\triangleleft}$ and the first inclusion is proved. The second inclusion $\big( K +F\big)^{\triangleleft} \subseteq K^{\triangleleft} \cap F^{\triangleleft}$ immediately follows from the fact that $0\in K\cap F$ (and hence $K\cup F\subseteq K+F$).

(iv) Let $\chi_1,\dots,\chi_n\in K$ and  $g\in V$. Then for every $i,j=1,\dots,n$, we have $(\chi_i, jg)=(\chi_i,g)^j\in \Ss_+$, and hence $(\chi_i,g)\in \left(\tfrac{1}{n}\right)\Ss_+$. Therefore
\[
(\chi_1+\dots+\chi_n, g)=(\chi_1,g)\cdots(\chi_n,g)\in \left(\tfrac{1}{n}\right)\Ss_+ \stackrel{n}{\cdots} \left(\tfrac{1}{n}\right)\Ss_+ =\Ss_+.
\]
Thus $g\in \big((n)K \big)^{\triangleleft}$.

(v) immediately follows from (iii).

(vi) First we prove the equality $\big((n)\qc(C)\big)^{\triangleright}=\big( (n)C \big)^{\triangleright}$ (the third equality of this clause can be proved analogously). The inclusion ``$\subseteq$'' is clear. To prove the inverse inclusion ``$\supseteq$'', let $\chi\in \big( (n)C \big)^{\triangleright}$. By (i) applied to $\Gsa$, we have $\chi\in \big( (n)C \big)^{\triangleright}$ if and only if $(\chi,g)\in \left(\tfrac{1}{n}\right)\Ss_+$ for every $g\in C$. But since $\left(\tfrac{1}{n}\right)\Ss_+$ is quasi-convex, the set $\chi^{-1} \big(\left(\tfrac{1}{n}\right)\Ss_+\big)$ is a quasi-convex subset of $G$ containing $C$. Therefore $\qc(C)\subseteq \chi^{-1} \big(\left(\tfrac{1}{n}\right)\Ss_+\big)$ and hence $(n)\qc(C)\subseteq \chi^{-1}(\Ss_+)$ as desired.

Now we prove the equality $\big( (n)C \big)^{\triangleright}= (C^{\triangleright})_{(n)}$. Let $\chi\in (C^{\triangleright})_{(n)} $ and $g\in C$. Then, for every $i=1,\dots,n$, we have $(i\chi,g)=(\chi,g)^i\in \Ss_+$ and hence $(\chi,g)\in \Ss_n$. Now, for every $g_1,\dots,g_n\in C$, we obtain
\[
(\chi, g_1+\dots+g_n)=(\chi,g_1)\cdots(\chi,g_n)\in \Ss_n \stackrel{n}{\cdots} \Ss_n =\Ss_+
\]
and hence  $\chi\in \big( (n)C\big)^{\triangleright}$. Thus $(C^{\triangleright})_{(n)} \subseteq \big( (n)C\big)^{\triangleright}$.
To prove the inverse inclusion  $(C^{\triangleright})_{(n)} \supseteq \big( (n)C\big)^{\triangleright}$, fix $\chi\in  \big( (n)C\big)^{\triangleright}$. Then, similar to (i), $(\chi,g)\in \Ss_n$ for every $g\in C$. Therefore $(i\chi,g)=(\chi,g)^i\in \Ss_+$  for every $g\in C$ and each $i=1,\dots,n$. Thus $\chi,2\chi,\dots,n\chi\in C^{\triangleright}$ and hence $\chi\in (C^{\triangleright})_{(n)}$.

Applying the first equality to $k=2$ and $C':=(n)C$, we obtain
\[
(C'+C')^\triangleright=((n)C+(n)C)^\triangleright=\big(\qc((n)C)+\qc((n)C)\big)^\triangleright
\]
and hence $\qc((n)C)+\qc((n)C)\subseteq \big(\qc((n)C)+\qc((n)C)\big)^{\triangleright\triangleleft} \subseteq \qc((2n)C)$. This proves the second inclusion.

To show that $X$ is a subgroup of $G$, first we note that $X$ is symmetric since all $\qc((n)C)$ are symmetric. Now, if $x,y\in X$, choose $n$ such that $x,y\in \qc((n)C)$. Then, by the second inclusion, $x+y\in \qc((2n)C)\subseteq X$. Thus $X$ is a subgroup of $G$.


(vii) Let $g\in (A^{\triangleleft})_{(n)} $ and $\chi\in A$. For every $i=1,\dots,n$, we have $(\chi,ig)=(\chi,g)^i\in \Ss_+$ and hence $(\chi,g)\in \Ss_n$. Now, for every $\chi_1,\dots,\chi_n\in A$, we obtain
\[
(\chi_1+\dots+\chi_n, g)=(\chi_1,g)\cdots(\chi_n,g)\in \Ss_n \stackrel{n}{\cdots} \Ss_n =\Ss_+
\]
and hence  $g\in \big( (n)A\big)^{\triangleleft}$. Thus $(A^{\triangleleft})_{(n)} \subseteq \big( (n)A\big)^{\triangleleft}$.

To prove the inverse inclusion  $(A^{\triangleleft})_{(n)} \supseteq \big( (n)A\big)^{\triangleleft}$, fix $g\in  \big( (n)A\big)^{\triangleleft}$. Then, by (i), $(\chi,g)\in \Ss_n$ for every $\chi\in A$. Therefore $(\chi,ig)=(\chi,g)^i\in \Ss_+$  for every $\chi\in A$ and each $i=1,\dots,n$. Thus $g,2g,\dots,ng\in A^{\triangleleft}$ and hence $g\in (A^{\triangleleft})_{(n)}$.

To show that $H:=\bigcap_{n\in\NN} (A^{\triangleleft})_{(n)}$ is a subgroup of $G$, we note first that $H$ is symmetric because so are all $(A^{\triangleleft})_{(n)}$. Now, let $h_1,h_2\in H$. Then for every $n\in\NN$, we have $h_1,h_2\in (A^{\triangleleft})_{(2n)}=\big( (2n)A\big)^{\triangleleft}$ and therefore, by (ii),
\[
h_1+h_2\in \big( (2n)A\big)^{\triangleleft}+\big( (2n)A\big)^{\triangleleft}=\big( (n)A+(n)A\big)^{\triangleleft}+\big( (n)A+(n)A\big)^{\triangleleft}\subseteq \big( (n)A\big)^{\triangleleft}=(A^{\triangleleft})_{(n)}.
\]
Hence $h_1+h_2\in H$. Thus $H$ is a subgroup of $G$.

(viii) Set $A:=C^\triangleright$. Since $C$ is quasi-convex we have $C=A^\triangleleft$.  For every $n\in\NN$, (vii) implies $C_{(n)}=\big( (n)A\big)^{\triangleleft}$ and hence $C_{(n)}$ is quasi-convex. Therefore, by (ii) applied to $K=(n)A$, $C_{(2n)}+C_{(2n)}\subseteq C_{(n)}$. The inclusion $(2^k)C_{(2^k n)}\subseteq C_{(n)}$ follows from $C_{(2n)}+C_{(2n)}\subseteq C_{(n)}$ by induction on $k\in\NN$. Finally, by (vii), $\bigcap_{n\in\NN} C_{(n)}=\bigcap_{n\in\NN} (A^{\triangleleft})_{(n)}$ is a subgroup of $G$.

(ix)The assertion follows from the next equivalences: $g\in (C\cap H)_{(n)}$ if and only if $g,2g,\dots,ng\in C\cap H$ if and only if $g,\dots,ng\in C$ and $g\in H$ (since $H$ is a subgroup of $G$) if and only if $g\in C_{(n)}$ and $g\in H$.

(x) By (vii), it suffices to prove that $(D^{\triangleleft_H})_{(n)}=H\cap(A^{\triangleleft})_{(n)}$. Let $\Id:H\to G$ be the identity inclusion. Since $A^\triangleleft$ is a quasi-convex subset of $G$, we obtain that $\Id^{-1}(A^\triangleleft)=A^\triangleleft \cap H$ is a quasi-convex subset of $H$. Therefore $D^{\triangleleft_H}=H\cap A^\triangleleft$. Thus, by (ix) applied to $C=A^\triangleleft$, we obtain $(D^{\triangleleft_H})_{(n)}=H\cap(A^{\triangleleft})_{(n)}$.

(xi) and (xii) can be proved as (ii) and (vii), respectively. \qed
%
%
\end{proof}

Let $G$ and $H$ be two abelian topological groups.  A family $\FF$  of continuous homomorphisms from $G$ to $H$ is called {\em equicontinuous} if for every neighborhood $V$ of zero in $H$ there is a neighborhood $U$ of zero in $G$ such that $f(U)\subseteq V$ for every $f\in \FF$. We shall use the following standard fact (since we do not know an exact reference we prove this result).

\begin{proposition} \label{p:Mackey-equi}
Let $(G,\tau)$ be an abelian topological  group, and let $K$ be a subset of $\widehat{G}$. Then $K$ is equicontinuous if and only if $K^{\triangleleft}$ is a neighborhood of zero. In this case $\qc_{\widehat{G}}\big((n)K\big)$ is equicontinuous for every $n\in\NN$.
\end{proposition}

\begin{proof}
Assume that $K$ is equicontinuous. Then for $V=\Ss_+$, there is $U\in\Nn(G)$ such that $\chi(U)\subseteq \Ss_+$ for every $\chi\in K$. This means that $U\subseteq K^{\triangleleft}$. Thus $K^{\triangleleft}$ is a neighborhood of zero. Conversely, assume that $K^{\triangleleft}$ is a neighborhood of zero. Since $K^{\triangleleft}=(K\cup\{0\})^{\triangleleft}$, we can assume that $0\in K$. For every $n\in\NN$, choose $U\in\Nn(G)$ such that $(n)U\subseteq K^{\triangleleft}$.  Then, by (i) and (iv) of Lemma  \ref{l:polar-1}, we obtain $(\chi,g)\in  \big(\tfrac{1}{n}\big)\Ss_+$ for every $\chi\in  K$ and each $g\in U$. Thus $K$ is equicontinuous.

To prove the last assertion, fix $n\in\NN$. Since $K^{\triangleleft}$ is a neighborhood of zero, (vii) of Lemma \ref{l:polar-1} implies that $\big((n)K\big)^\triangleleft= \big(K^\triangleleft\big)_{(n)}$ is also a neighborhood of zero. Hence $\qc_{\widehat{G}}\big((n)K\big)=\big((n)K\big)^{\triangleleft\triangleright}$ is equicontinuous. \qed
\end{proof}

The dual group $\widehat{G}$ of an abelian topological group $G$ endowed with the compact-open topology $\tau_k$ is denoted by $G^{\wedge}$.
We shall use the next folklore fact, for a more general assertion see Proposition 3.5 of \cite{Aus}  (nevertheless we add its short proof for the sake of completeness).

\begin{proposition} \label{p:polar-U-k-compact}
Let $(G,\tau)$ be an abelian topological group. Then for every neighborhood $U$ of zero, the polar $U^{\triangleright}$ is a compact subset of $G^\wedge$. Consequently,  $ U^{\triangleright}$ is also $\sigma(\widehat{G},G)$-compact.
\end{proposition}

\begin{proof}
Denote by $\CC(G,\Ss)$  the space $C(G,\Ss)$  of all continuous functions from $G$ to $\Ss$ endowed with the compact-open topology $\tau_k$. It is clear that $G^\wedge$ is a closed subspace of $\CC(G,\Ss)$. By Proposition \ref{p:Mackey-equi}, the set $U^{\triangleright}$ is equicontinuous and hence evenly continuous. Clearly, $U^{\triangleright}$ is also closed in the pointwise topology on $C(G,\Ss)$. Now (the sufficiency in) the Ascoli theorem \cite[Theorem~3.4.20]{Eng} implies that $U^{\triangleright}$ is a compact subset of $\CC(G,\Ss)$ and hence of $G^\wedge$. The last assertion follows from the fact $\sigma(\widehat{G},G)\leq \tau_k$. \qed
\end{proof}

Let $(G,\tau)$ be an abelian topological group. The homomorphism $\alpha_G : G\to G^{\wedge\wedge} $, $g\mapsto (\chi\mapsto \chi(g))$, is called {\em the canonical homomorphism}. If $\alpha_G$ is a topological isomorphism, the group $G$ is called (Pontryagin) {\em  reflexive}. Every locally compact abelian (LCA) group is reflexive, and each reflexive group is locally quasi-convex. We shall use the next assertion.
\begin{proposition} \label{p:metriz-Baire-group}
Let $G$ be a metrizable lqc abelian group. Then $G$ is locally precompact if and only if $G^\wedge$ has the Baire property. In this case $G^\wedge$ is an LCA group.
\end{proposition}

\begin{proof}
Assume that $G$ is locally precompact, so a completion $\overline{G}$ of $G$ is an LCA group. By \cite{Aus} or \cite{Cha}, $G^\wedge=\overline{G}^\wedge$ and hence, by the Pontryagin Duality Theorem,  $G^\wedge$ is locally compact and hence it has the Baire property.

Conversely, assume that $G^\wedge$ has the Baire property.
Let $\{U_n\}_{n\in\NN}\subseteq \Nn_{qc}(G)$ be a base at zero such that $U_{n+1}+U_{n+1}\subseteq U_n$ for all $n\in\NN$. It is proved in \cite{Cha} that the polars $U_n^{\triangleright}$ $(n\in\NN)$ swallow the compact subsets of $G^\wedge$ and are also compact in $G^\wedge$ by Proposition \ref{p:polar-U-k-compact}. Since $G^\wedge=\bigcup_{n\in\NN} U_n^{\triangleright}$, by the Baire property, there is $n\in\NN$ such that $U_n^{\triangleright}$ is a closed and compact neighborhood of a point in $G^\wedge$. Thus $G^\wedge$ is a locally compact group. By the Pontryagin Duality Theorem, the bidual group $G^{\wedge\wedge}$ is also locally compact. Since $G$ is locally quasi-convex and metrizable, the canonical map $\alpha_G$ is an embedding (see 5.12 and 6.10 in \cite{Aus}). Thus $G$ is a locally precompact group.\qed
\end{proof}

A subgroup $H$ of an abelian topological group $G$ is called {\em dually closed} in $G$ if for every $g\in G\SM H$ there exists  $\chi \in H^\perp$ such that $(\chi,g)\not= 1$, i.e., $H=\cap_{\chi\in H^\perp} \ker(\chi)$. Clearly, every dually closed subgroup is closed. A subgroup $H$ is called {\em dually embedded} in $G$ if every continuous character of $H$ can be extended to a  continuous character of $G$. The next result is folklore.

\begin{proposition} \label{p:precompact-dc-de}
Every $($closed$)$ subgroup $H$  of a precompact abelian group $G$ is dually embedded $($and dually closed$)$ in $G$.
\end{proposition}

\begin{proof}
Let $\overline{G}$ be a completion of $G$. Then $\overline{G}$ and the closure $\overline{H}$ of $H$ in $\overline{G}$ are compact abelian groups. By Corollary 24.12 of \cite{HR1}, $\overline{H}$ is a dually closed and dually embedded subgroup of $\overline{G}$. Since $\widehat{\overline{H}}=\widehat{H}$ and $\widehat{\overline{G}}=\widehat{G}$, it follows that $H$ is dually embedded in $G$.
If in addition $H$ is closed, then $H=\overline{H}\cap G$ and hence $H$ is dually closed in $G$.\qed
\end{proof}

Recall that a topological group $(G,\tau)$ is {\em minimal} if there is no coarser (Hausdorff) topological group topology for $G$. Clearly, every compact group is minimal. Below we give a nice characterization of dual closeness.
\begin{proposition} \label{p:minmal-dual}
Let $H$ be a subgroup of a $MAP$ abelian group $(G,\tau)$. Then:
\begin{enumerate}
\item[{\rm(i)}] $H$ is $\sigma(G,\widehat{G})$-closed if and only if $H$ is a dually closed subgroup of $G$ $($in particular, $H$ is a quasi-convex subset of $G$$)$.
\item[{\rm(ii)}]  If $H$ satisfies (i) and is in addition minimal  $($for example, $H$ is compact$)$, then $H$ is also dually embedded in $G$.
\end{enumerate}
\end{proposition}

\begin{proof}
(i) Assume that $H$ is $\sigma(G,\widehat{G})$-closed. Since $\sigma(G,\widehat{G})$ is a precompact topology on $G$, Proposition \ref{p:precompact-dc-de} implies that $H$ is  dually closed in $G$. Conversely, if $H$ is a dually closed subgroup of $G$, then $H=\cap_{\chi\in H^\perp} \ker(\chi)$ and hence it is $\sigma(G,\widehat{G})$-closed.

(ii) Assume that $H$ is minimal. Then $(H,\tau{\restriction}_H)$ is topologically isomorphic to $(H,\sigma(G,\widehat{G}){\restriction}_H)$. Applying  Proposition \ref{p:precompact-dc-de} to the last group we obtain that $H$ is dually embedded in $G$. \qed
\end{proof}

\begin{remark} \label{rem:dually-embed}
Note that (i) of Proposition \ref{p:minmal-dual} means that the property of being a dually closed subgroup depends only on the duality $(G,\widehat{G})$: $H$ is dually closed in $(G,\tau)$ if and only if $H$ is dually closed in $(G,\nu)$ for any group topology $\nu$ on $G$ compatible with $\tau$ (because in any case $H$ is $\sigma(G,\widehat{G})$-closed). However, in general the property of being a dually embedded subgroup depends also on the topology $\tau$. Indeed, if $H$ is {\em not} a dually embedded subgroup of $(G,\tau)$ (see for example Remark 5.27 of \cite{Aus}), nevertheless,  by Proposition \ref{p:precompact-dc-de}, $H$ is a dually embedded subgroup of $\big(G,\sigma(G,\widehat{G})\big)$.\qed
\end{remark}
We complement Proposition \ref{p:minmal-dual} and Remark \ref{rem:dually-embed} by the next assertion.

\begin{proposition} \label{p:dually-embed}
Let $H$ be a  dually embedded subgroup of a $MAP$ abelian group $(G,\tau)$. Then $H$ is  a  dually embedded subgroup of  $(G,\nu)$ for any group topology $\nu$ on $G$ compatible with $\tau$ such that $\sigma(G,\widehat{G})\leq \nu\leq\tau$.
\end{proposition}

\begin{proof}
Let $\chi\in \widehat{(H,\nu{\restriction}_H)}$. Since $\nu\leq\tau$ it follows that $\chi\in \widehat{(H,\tau{\restriction}_H)}$. Therefore there is $\eta\in \widehat{G}$ such that $\eta{\restriction}_H=\chi$. As $\sigma(G,\widehat{G})\leq \nu$ we obtain $\eta\in \widehat{(G,\nu)}$. Thus $\eta$ extends $\chi$ and hence $H$ is  dually embedded in $(G,\nu)$.\qed
\end{proof}

We shall use the following well known facts. 

\begin{proposition} \label{p:basis-weak}
Let $(G,\tau)$ be an abelian topological group.
\begin{enumerate}
\item[{\rm(i)}] The sets of the form $V=F^{\triangleleft}$, where $F\subseteq \widehat{G}$ is finite, form a base of the weak topology $\sigma(G,\widehat{G})$.
\item[{\rm(ii)}] An element $\chi\in G^\ast$ is continuous if and only if there is $U\in\Nn(G)$ such that $\chi\in U^{\triangleright}$.
\end{enumerate}
\end{proposition}

\begin{proof}
(i) follows from Propositions 3.4(i) of \cite{Aus} applying to the group $\Gsa$, and (ii) is  Proposition 3.4(ii) of \cite{Aus}.\qed 
\end{proof}
Since, by  \cite[Theorem~7.11]{Aus}, $F^{\triangleleft\triangleright}$ is finite for every finite $F\subseteq \widehat{G}$, we can always assume that $F$ is quasi-convex and $V=F^{\triangleleft}$ is a quasi-convex weak neighborhood of $0\in G$ (in this case $F=V^{\triangleright}$).

Let $G$ and $H$ be abelian topological groups, and let $p:G\to H$ be a  continuous homomorphism. The {\em adjoint map} $p^\ast: \widehat{H}\to \widehat{G}$ is defined by $p^\ast(\chi)(g):=\big(\chi,p(g)\big)$ for all $\chi\in \widehat{H}$ and $g\in G$. If $F\subseteq G$ and $\e>0$, set
\[
[F;\e]:=\{ \chi\in \widehat{G}: |(\chi,g)-1|<\e \; \mbox{ for every }\; g\in F\}.
\]

\begin{lemma} \label{l:adjoint}
Let $(G,\tau)$ and $(H,\nu)$ be $MAP$ abelian topological groups, $p:H\to G$ be a continuous homomorphism, and let $p^\ast: \widehat{G}_{\sigma^\ast} \to \widehat{H}_{\sigma^\ast} $ be the adjoint map of $p$. Then:
\begin{enumerate}
\item[{\rm(i)}] $p^\ast$ is continuous.
\item[{\rm(ii)}] If $p$ is surjective, then $p^\ast$ is an embedding.
\item[{\rm(iii)}] Let $p$ be a surjective map, and let the topology $\tau$ and the quotient topology $\nu/\ker(p)$ be compatible (for example, $p$ is a quotient map). Then $p^\ast$ is an embedding and $p^\ast(\widehat{G})$ is a dually embedded and dually closed subgroup of $\widehat{H}_{\sigma^\ast}$. Moreover, if $A\subseteq \widehat{G}_{\sigma^\ast}$ is quasi-convex, then so is $p^\ast(A)\subseteq \widehat{H}_{\sigma^\ast}$.
\end{enumerate}
\end{lemma}

\begin{proof}
(i) Let $F$ be a finite subset of $H$, and let $\e>0$. Then for every $\eta\in [p(F);\e]$ and each $h\in F$, we obtain
$
\big| (p^\ast(\eta),h)-1\big| =|(\eta,p(h))-1|<\e
$
and hence
$
p^\ast\big( [p(F);\e]\big)\subseteq [F;\e].
$
Thus $p^\ast$ is continuous.

(ii) It is clear that $p^\ast$ is injective. Therefore, by (i), to show that $p^\ast$ is an embedding it suffices to prove that $p^\ast$ is relatively open. To this end, let $\widetilde{F}$ be a finite subset of $G$ and $\e>0$. Choose a finite subset $F$ of $H$ such that $p(F)=\widetilde{F}$. Then, for every $\chi=p^\ast(\eta)$ where $\eta\in \widehat{G}$ and each $h\in F$, we have
\[
|(\chi,h)-1|=\big| (p^\ast(\eta),h)-1\big| =|(\eta,p(h))-1|.
\]
Therefore $[F;\e]\cap p^\ast(\widehat{G}) \subseteq p^\ast\big([\widetilde{F};\e]\big)$, and hence $p^\ast$ is relatively open.
%
%

(iii) Taking into account Proposition \ref{p:precompact-dc-de} and (ii), it is sufficient to show that the image $p^\ast(\widehat{G})$ is closed in $\widehat{H}_{\sigma^\ast}$. Let $\{\chi_i\}_{i\in I}\subseteq \widehat{G}$ be such that $p^\ast(\chi_i)\to \chi$ in $\widehat{H}_{\sigma^\ast}$. For every $h\in \ker(p)$, we have $\big(p^\ast(\chi_i),h\big)=(\chi_i,p(h))=1$. Therefore $\chi(\ker(p))=\{1\}$ and hence $\chi$ induces a continuous  character $\widetilde{\chi}$ of the quotient group $H/\ker(p)$ by the rule $(\widetilde{\chi}, h+\ker(p))=(\chi,h)$. Since $\tau$ and $\nu/\ker(p)$ are compatible, $\widetilde{\chi}\in \widehat{G}$. As
\[
\big(p^\ast(\widetilde{\chi}),h\big)=(\widetilde{\chi},p(h))=(\widetilde{\chi}, h+\ker(p))=(\chi,h) \quad (h\in H)
\]
we obtain $\chi=p^\ast(\widetilde{\chi})\in p^\ast(G)$, as desired.

To prove the last assertion, assume that $A\subseteq \widehat{G}_{\sigma^\ast}$ is quasi-convex. Let $\eta\in \widehat{H}\SM p^\ast(A)$. If $\eta\not\in p^\ast(\widehat{G})$, then the dual closeness of $p^\ast(\widehat{G})$ implies that there is $h\in \big(p^\ast(\widehat{G})\big)^\top$ such that $(\eta,h)\not=1$. Then there is $n\in\NN$ such that $(\eta,h)^n=(n\eta,h)\not\in\Ss_+$. It remains to note that $n\eta\in \big(p^\ast(\widehat{G})\big)^\top$. Assume now that $\eta\in p^\ast(\widehat{G})$, so $\eta=p^\ast(\chi)$ for some $\chi\in \widehat{G}$. Since $A$ is quasi-convex, there is $g\in A^\triangleleft$ such that $(\chi,g)\not\in\Ss_+$. Then $h:=p(g)$ satisfies  $(\eta,h)=(\chi,g)\not\in\Ss_+$ and, for every $\xi=p^\ast(\zeta)\in p^\ast(A)$  with $\zeta\in A$, we have $(\xi,h)=(\zeta,g)\in\Ss_+$, as desired. Thus $p^\ast(A)$ is quasi-convex.\qed
\end{proof}

For an lcs $E$, the value of $\chi\in E'$ on $x\in E$ is denoted by $\langle\chi,x\rangle$ or $\chi(x)$. For a subset $B$ of $E$, the set
\[
B^\circ:=\{\chi\in E': |\langle \chi,x\rangle|\leq 1 \; \mbox{ for all }\; x\in B\}
\]
is called the {\em polar} of $B$. Clearly, $B^\circ$ is arc-connected. If $B$ is closed and absolutely convex, then, by the Bipolar theorem, $B={}^\circ(B^{\circ})$. For every $a\in \IR$, we set $aB:=\{ax: x\in B\}$. Define
\[
\psi: E' \to \widehat{E},  \; \psi(\chi):= e^{2\pi i \chi},  \quad \big( \mbox{i.e. } \psi(\chi)(x):= \exp\{2\pi i \langle\chi,x\rangle\} \mbox{ for } x\in E\big).
\]

A proof of the next important result can be found in \cite[Proposition~2.3]{Ban}.
\begin{fact} \label{f:dual-E}
Let $E$ be an lcs and let $\psi: E' \to \widehat{E},  \psi(\chi):= e^{2\pi i \chi}$. Then:
\begin{enumerate}
\item[{\rm (i)}] $\psi$ is an algebraic isomorphism;
\item[{\rm (ii)}] $\psi$ is a topological isomorphism of $(E', \tau_k)$ onto $E^\wedge$.
\end{enumerate}
\end{fact}
We shall say that $\psi$ is the {\em canonical isomorphism } of $E'$ onto $\widehat{E}$.


\begin{lemma} \label{l:polar-neigh}
Let $B$ be a closed and absolutely convex subset of an lcs $E$, and let $K:=B^\circ$. Then,  for every $n\in\NN$, we have
\[
\begin{aligned}
\tfrac{1}{4n}B & =\big(\psi(nK)\big)^{\triangleleft},\;\;\; B^{\triangleright}=\psi(\tfrac{1}{4}B^\circ)=\psi(\tfrac{1}{4}K), \\  \big(\tfrac{1}{4n}B\big)^{\triangleright} & =\psi(nK)=n\psi(K)=(n)\psi(K), \; \mbox{ and }\; \psi\big(\tfrac{1}{4n} B^\circ\big) = \big((n)B\big)^\triangleright.
\end{aligned}
\]
In particular, $B$ and $\psi(K)$ are quasi-convex.
\end{lemma}

\begin{proof}
Since
$
\big(\psi(nK)\big)^{\triangleleft}=\left\{ x\in E: \exp\{ 2\pi i \langle n\chi,x\rangle\} \in \Ss_+ \mbox{ for all } \chi\in K\right\},
$
it follows that $x\in \big(\psi(nK)\big)^{\triangleleft}$ if and only if
\[
\langle n\chi,x\rangle\in \left[-\tfrac{1}{4},\tfrac{1}{4}\right] +\ZZ, \quad \mbox{ for all }\; \chi\in K.
\]
But since $K$ is arc-connected and contains zero, this is possible if and only if $\langle4n\chi,x\rangle\in \left[-1,1\right]$ for all $\chi\in K$, i.e. $x\in \tfrac{1}{4n}B$.

As $B^{\triangleright}=\left\{ \chi\in \widehat{E}: \exp\{ 2\pi i \langle \psi^{-1}(\chi),x\rangle\} \in \Ss_+ \mbox{ for all } x\in B\right\}$, we obtain that $\chi\in B^{\triangleright}$ if and only if
\[
\langle \psi^{-1}(\chi),x\rangle\in \left[-\tfrac{1}{4},\tfrac{1}{4}\right] +\ZZ, \quad \mbox{ for all }\; x\in B,
\]
and since $B$ is arc-connected and contains zero, $\chi\in B^{\triangleright}$ if and only if $\langle\psi^{-1}(\chi),4x\rangle\in \left[-1,1\right]$ for all $x\in B$, i.e. $\chi\in \psi(\tfrac{1}{4}B^\circ)$. Thus $B^{\triangleright}=\psi(\tfrac{1}{4}B^\circ)=\psi(\tfrac{1}{4}K)$.

The third equality $\big(\tfrac{1}{4n}B\big)^{\triangleright} =\psi(nK)$ can be proved analogously, and the equalities $\psi(nK)=n\psi(K)=(n)\psi(K)$ hold because $\psi$ is an isomorphism and  $nK=(n)K$.

To prove the last equality, let $\chi\in \tfrac{1}{4n} B^\circ$. Then for every $x_1,\dots,x_n\in B$, we obtain $|\langle\chi,x_k\rangle|\leq \tfrac{1}{4n}$ and hence
\[
\big(\psi(\chi),x_1+\cdots+x_n\big)=\exp\Big\{2\pi i\sum_{k=1}^n \langle\chi,x_k\rangle\Big\}\in\Ss_+.
\]
Thus $\psi(\chi)\in \big((n)B\big)^\triangleright$. Conversely, if $\psi(\chi)\in \big((n)B\big)^\triangleright$, then for every $k=1,\dots,n$, and each $x\in B$ we have
\[
(\psi(\chi),kx)=(\psi(\chi),x)^k =\exp\{2\pi ik \langle\chi,x\rangle\} \in \Ss_+
\]
and hence $\langle\chi,x\rangle\in [-\tfrac{1}{4n},\tfrac{1}{4n}]+\ZZ$ for every $x\in B$. But since $B$ is arc-connected and contains zero, we obtain $\langle\chi,x\rangle\in [-\tfrac{1}{4n},\tfrac{1}{4n}]$ for every $x\in B$, and hence $\chi\in \tfrac{1}{4n}B^\circ$.\qed
\end{proof}


\section{Main results}  \label{sec:compatible-first}


The main result of this section gives the first description of all compatible group topologies on an lqc abelian group, see Theorem \ref{t:compatible-description}. Using this description we characterize those lqc abelian groups which admits a Mackey group topology (Theorem \ref{t:Mackey-exist}) or are Mackey groups (Theorem  \ref{t:Mackey-group-equi}). This description  is based on the construction of {\em polar} lqc group topologies  described in the next lemma.

\begin{lemma} \label{l:topology-G-dual}
Let $(G,\tau)$ be a $MAP$ abelian group, and let $K$ be a quasi-convex subset of $\Gsa$.
\begin{enumerate}
\item[{\rm(i)}] The family  $\U_K^\tau$ of subsets of $G$ of the form
\[
\big( (n)K \big)^{\triangleleft} \cap U, \; \mbox{ where }\; n\in\NN \; \mbox{ and }\; U\in \Nn_{s}(G,\tau),
\]
defines a base at zero of a $($Hausdorff$)$ group topology $\TTT_K^\tau$ on $G$ which is finer than $\tau$. If additionally $G$ is locally quasi-convex, then the sets
\[
\big( (n)K \big)^{\triangleleft} \cap U, \; \mbox{ where }\; n\in\NN \; \mbox{ and }\; U\in \Nn_{qc}(G,\tau),
\]
define a base at zero of an lqc  group topology $\TTT_K^\tau$ on $G$ which is finer than $\tau$.
\item[{\rm(ii)}] $\TTT^\tau_K =\tau$ if and only if $K=U^{\triangleright}$ for some $\tau$-neighborhood $U$ of zero.
\item[{\rm(iii)}] The family  $\U_K$ of subsets of $G$ of the form
\[
\big((n)K+F\big)^{\triangleleft}, \; \mbox{ where }\; n\in\NN \; \mbox{ and }\; 0\in F\subseteq \widehat{G} \mbox{ is finite},
\]
defines  a base at zero of the lqc group topology $\TTT_K:=\TTT_K^{\,\sigma(G,\widehat{G})}$ on $G$. Consequently,
\[
\sigma(G,\widehat{G})\leq \TTT_K \leq \TTT_K^\tau.
\]
\item[{\rm(iv)}] $\sigma(G,\widehat{G})=\vee \{\TTT_F: 0\in F\subseteq \widehat{G} \mbox{ is finite and quasi-convex}\}$.

\item[{\rm(v)}] Set $H:=K^\top$, and let $p:G\to G/H$ be the quotient map and $\widetilde{K}:=\big(p^\ast\big)^{-1}(K)\subseteq \widehat{G/H}$. Then:

 (a) $K\subseteq p^\ast(\widehat{G/H})$, $\widetilde{K}$ is quasi-convex, and $\widetilde{K}$ is compact if and only if so is $K$;

 (b) the topology $\nu:=\TTT_{\widetilde{K}}^{\,\sigma(G/H,\widehat{G/H})}$ on $G/H$ is weaker than the quotient topology $\TTT_K/H$.
\item[{\rm(vi)}]  In the notations of {\rm(v)}, the countable family  $\V_{\widetilde{K}}=\big\{ \big((n)\widetilde{K}\big)^{\triangleleft}: n\in\NN\big\}$ of subsets of $G/H$ defines an lqc metrizable 
    group topology $\nu_{\widetilde{K}}$ on $G/H$ such that $\nu_{\widetilde{K}}\leq \nu$.
\end{enumerate}
\end{lemma}

\begin{proof}
(i) By \cite[Prop.~1,\S 1.2, Chapter~III]{Bourbaki}, to show that the family $\U_K^\tau$ defines a group topology on $G$  it is sufficient to check that  $\U_K^\tau$ satisfies the following conditions:
\begin{enumerate}
\item[(a)] $0\in W$ for every $W\in \U_K^\tau$;
\item[(b)] for every $V\in \U_K^\tau$ there is $W\in \U_K^\tau$ such that  $-W\subseteq V$;
\item[(c)] for every $V\in \U_K^\tau$ there is $W\in \U_K^\tau$ such that  $W+W\subseteq V$;
\item[(d)] for every $V,W\in \U_K^\tau$ there is $\mathcal{O}\in \U_K^\tau$ such that  $\mathcal{O}\subseteq V\cap W$.
\end{enumerate}
The conditions (a) and (b) hold because, by definition, all elements of $\U_K^\tau$ contain $0$ and are symmetric. To check (d), let $V=\big( (n_1)K \big)^{\triangleleft} \cap U_1$ and $W=\big( (n_2)K \big)^{\triangleleft} \cap U_2$. Then the set $\mathcal{O}:= \big( (n_1+n_2)K \big)^{\triangleleft} \cap (U_1\cap U_2)\in \U_K^\tau$ trivially satisfies  $\mathcal{O}\subseteq V\cap W$.

Now we check (c). Let $V=\big( (n)K \big)^{\triangleleft} \cap U$, where $n\in\NN$ and $U\in \Nn_{s}(G)$. Choose $U_0\in \Nn_{s}(G)$ such that $U_0+U_0\subseteq U$, and set $W:=\big( (2n)K \big)^{\triangleleft} \cap U_0$. Then, by (ii) of Lemma \ref{l:polar-1} applied to $(n)K$ instead of $K$, we obtain $W+W\subseteq V$ as desired.

By construction $\tau\leq \TTT^\tau_K$, and hence $\TTT^\tau_K$ is Hausdorff.

Assume that $G$ is locally quasi-convex. Since the inverse polars are quasi-convex, and the intersection of any family of quasi-convex sets is also quasi-convex, the topology $\TTT^\tau_K$ is locally quasi-convex.

(ii) Assume that $\tau= \TTT^\tau_K$. Then the set $U:=K^{\triangleleft}$ must be also a $\tau$-neighborhood of zero, and hence $K=K^{\triangleleft\triangleright}=U^{\triangleright}$. Conversely, assume that $K=U^{\triangleright}$ for some  $\tau$-neighborhood $U$ of zero and let $n\in\NN$. We can assume that $U=K^{\triangleleft}$. Choose $V\in \Nn_{s}(G,\tau)$ such that $(n)V\subseteq U$. Then, by (iv) of Lemma \ref{l:polar-1}, we have $V\subseteq \big((n)K\big)^{\triangleleft}$. Thus $\TTT^\tau_K\leq \tau$ and hence $\tau= \TTT^\tau_K$.
\smallskip

(iii) follows from Lemma \ref{l:polar-1}(iii), Proposition  \ref{p:basis-weak}(i) and (i). The inclusions  $\sigma(G,\widehat{G})\leq \TTT_K \leq \TTT_K^\tau$ immediately follow from (i) and the fact $\sigma(G,\widehat{G})\leq\tau$. 

(iv) The inclusion $\sigma(G,\widehat{G})\leq \vee \TTT_F$ follows from the inclusion $\sigma(G,\widehat{G})\leq \TTT_F$, see (iii). To prove the inverse inclusion, let
\[
U=\big( (n_1)\widetilde{F}_1 +F_1\big)^{\triangleleft}\cap\cdots\cap \big( (n_k)\widetilde{F}_k +F_k\big)^{\triangleleft}
\]
be a standard neighborhood at zero in $\vee\TTT_F$. Set $F:=\bigcup_{i=1}^k \big( (n_i)\widetilde{F}_i +F_i\big)$. Then $F$ is finite, $0\in F$ and $F^{\triangleleft} \subseteq U$. Thus, by Proposition  \ref{p:basis-weak}(i), $\vee\TTT_F\leq \sigma(G,\widehat{G})$.

(v) By Lemma \ref{l:adjoint}, the adjoint map $p^\ast: (\widehat{G/H})_{\sigma^\ast} \to \Gsa$ is an embedding. To show  that $K\subseteq p^\ast(\widehat{G/H})$, let $\chi\in K$. Then $H\subseteq\ker(\chi)$ and hence $\chi$ induces a continuous character $\widetilde{\chi}$ of $G/H$ such that $\widetilde{\chi}\circ p=\chi$, i.e., $\chi=p^\ast(\widetilde{\chi})\in p^\ast(\widehat{G/H})$ as desired. Consequently, $K$ is compact if and only if so is $\widetilde{K}$. Also it is clear that $p^\ast(\widehat{G/H})=H^\perp =K^{\top\perp}$ (so $p^\ast$ has closed image, see also (iii) of Lemma \ref{l:adjoint}). To prove that $\widetilde{K}$ is quasi-convex, let $\chi\in\widehat{G/H}\SM \widetilde{K}$ and hence $p^\ast(\chi)\in\widehat{G}\SM K$. Since $K$ is quasi-convex, there is $g_0\in K^{\triangleleft}$ such that $(p^\ast(\chi),g_0)=(\chi,p(g_0)) \not\in \Ss_+$. Observe also that for every $\eta\in \widetilde{K}$, we have $(\eta,p(g_0))=(p^\ast(\eta),g_0)\in \Ss_+$. Therefore $p(g_0)\in \widetilde{K}^{\triangleleft}$ and $(\chi,p(g_0)) \not\in \Ss_+$. Thus $\widetilde{K}$ is quasi-convex and the clause (a) is proved.

To show that $\nu \leq\TTT_K/H $, fix arbitrarily an $n\in\NN$ and a finite $0\in \widetilde{F}\subseteq \widehat{G/H}$. Set $F:=p^\ast(\widetilde{F})$. We show that $p\big(\big((n)K+F\big)^{\triangleleft}\big)\subseteq \big((n)\widetilde{K}+\widetilde{F}\big)^{\triangleleft}$. Let $g\in \big((n)K+F\big)^{\triangleleft}$, $\eta_1,\dots,\eta_n\in \widetilde{K}$ and $\widetilde{f}\in\widetilde{F}$ be arbitrary. Then
\[
\big(\eta_1+\dots+\eta_n+\widetilde{f},p(g)\big)=\big(p^\ast(\eta_1)+\dots+p^\ast(\eta_n)+p^\ast(\widetilde{f}),g\big)\in\Ss_+
\]
and hence $p(g)\in \big((n)\widetilde{K}+\widetilde{F}\big)^{\triangleleft}$ as desired. Thus $\nu \leq\TTT_K/H $.

(vi) Analogously to (i), one can show that $\nu_{\widetilde{K}}$ is an lqc group topology. Clearly, $\nu_{\widetilde{K}}\leq \nu$. It remains to prove that $\nu_{\widetilde{K}}$ is Hausdorff. Let $0\not= h\in G/H$. Choose $g\in G\SM H$ such that $p(g)=h$. As $H=K^\top$ there is $\chi\in K$ such that $(\chi,g)\not= 1$. So there is an $n\in\NN$ such that $(n\chi,g)\not\in \Ss_+$. Then, for $\widetilde{\chi}:=\big(p^\ast\big)^{-1}(\chi)\in \widetilde{K}$, we have $\big(n\widetilde{\chi},p(g)\big)=(n\chi,g)\not\in \Ss_+$. But this means that $h=p(g)\not\in \big((n)\widetilde{K}\big)^{\triangleleft}$. Thus $\nu_{\widetilde{K}}$ is Hausdorff.\qed
\end{proof}

Since the topologies of the form $\TTT_K$ will play a crucial role for our description of compatible lqc topologies on lqc groups, it will be important to compare two such polar topologies.

\begin{lemma} \label{l:topology-G-dual-weak}
Let $(G,\tau)$ be a $MAP$ abelian group, and let $K$ and $K'$ be quasi-convex subsets of $\Gsa$.  Then:
\begin{enumerate}
\item[{\rm(i)}] $\TTT_{K'} \leq \TTT_K$ if and only if there are $n\in\NN$ and a finite $0\in F\subseteq \widehat{G}$ such that $K'\subseteq \qc_{\widehat{G}}\big((n)K+F\big)$. Consequently, $\TTT_{K'} =\TTT_K$  if and only if there are $n\in\NN$ and a finite $0\in F\subseteq \widehat{G}$ such that
    \[
    K'\subseteq \qc_{\widehat{G}}\big((n)K+F\big)\;\; \mbox{ and } \;\; K\subseteq \qc_{\widehat{G}}\big((n)K'+F\big).
    \]
\item[{\rm(ii)}] $\TTT_{\qc_{\widehat{G}}\big((n)K+F\big)} =\TTT_K$ for all $n\in\NN$ and a finite $0\in F\subseteq \widehat{G}$; in particular, $\TTT_{\qc_{\widehat{G}}\big((n)K\big)}=\TTT_K$.
\end{enumerate}
\end{lemma}

\begin{proof}
(i) Assume that  $\TTT_{K'} \leq \TTT_K$. Since ${K'}^{\triangleleft}$ is a $\TTT_{K'}$-neighborhood of zero, there are $n\in\NN$ and a finite $0\in F\subseteq \widehat{G}$ such that $\big((n)K+F\big)^{\triangleleft} \subseteq {K'}^{\triangleleft}$. Thus $K'=(K')^{\triangleleft\triangleright} \subseteq \big((n)K+F\big)^{\triangleleft\triangleright}= \qc_{\widehat{G}}\big((n)K+F\big)$ as desired.

To prove the sufficiency we recall (see (i) and (iii) of Lemma \ref{l:topology-G-dual}) that standard neighborhoods of zero in $\TTT_{K'}$ have the form $\big((m)K'\big)^{\triangleleft}\cap U$, where $U$ is a $\sigma(G,\widehat{G})$-neighborhood of zero and $m\in\NN$. Now, by (vi) of Lemma \ref{l:polar-1} applied to $A=(n)K+F$, we obtain
\[
\big((mn)K+(m)F\big)^{\triangleleft} =\Big((m)\big((n)K+F\big)\Big)^{\triangleleft} =\Big((m)\qc_{\widehat{G}}\big((n)K+F\big)\Big)^{\triangleleft} \subseteq \big((m)K'\big)^{\triangleleft}.
\]
Therefore $\big((m)K'\big)^{\triangleleft}\cap U$ is a $\TTT_K$-neighborhood of zero. Thus $\TTT_{K'} \leq \TTT_K$.

(ii) immediately follows from (i).\qed
\end{proof}

Considering the polar topologies of the form $\TTT_K$ by modulo the Mackey problem for an lqc abelian group $(G,\tau)$, it is natural to characterize those quasi-convex subsets $K$ of $\Gsa$ for which the topology $\TTT_K$ is compatible with the original topology $\tau$ on $G$. In Theorem \ref{t:Mackey-compatible-dual} below we show that $\TTT_K$ is compatible with $\tau$ if and only if $K$ belongs to the  following class of quasi-convex subsets of $\widehat{G}$.
\begin{definition} \label{def:Mac(G)}
For a $MAP$ abelian group $(G,\tau)$, denote by $\Mac(G)$ the family of all quasi-convex compact subsets $K$ of $\Gsa$ such that for every finite subset $0\in F\subseteq \widehat{G}$ and each $n\in\NN$, the subset $\big( (n)K +F\big)^{\triangleleft\blacktriangleright}$ of the group $G^\ast$ is contained in $\widehat{G}$.\qed
\end{definition}

Since the group $\Gsa$ and the operation of taking polars depend only on the duality $(G,\widehat{G})$ we obtain
\begin{proposition} \label{p:Mac-duality}
For a $MAP$ abelian group $G$, the family $\Mac(G)$ depends only on the duality $(G,\widehat{G})$.
\end{proposition}



The next theorem shows the topology $\TTT_K$ is compatible with $\tau$ exactly when $K\in\Mac(G)$.
\begin{theorem} \label{t:Mackey-compatible-dual}
Let $(G,\tau)$ be a $MAP$ abelian group, and let $K$ be a quasi-convex subset of $\Gsa$. Then the topology $\TTT_K$ is compatible with $\tau$ if and only if $K\in \Mac(G)$.
\end{theorem}

\begin{proof}
Assume that $\TTT_K$ is compatible with $\tau$, i.e. $\widehat{(G,\TTT_K)}=\widehat{G}$. First we show that $K$ is $\sigma(\widehat{G},G)$-compact. To this end, we recall that $K^{\triangleleft}$ is a neighborhood of zero in $\TTT_K$. Therefore, by Proposition \ref{p:polar-U-k-compact}, $K=K^{\triangleleft\triangleright}$ is a compact subset in $\sigma\big(\widehat{(G,\TTT_K)},G\big)=\sigma(\widehat{G},G)$. Thus $K$ is $\sigma(\widehat{G},G)$-compact. Now, fix an $n\in\NN$ and a finite subset $0\in F$ of $\widehat{G}$. We claim that every $\eta\in \big( (n)K +F\big)^{\triangleleft\blacktriangleright}$ belongs to $\widehat{G}$. Indeed, since $U:=\big( (n)K +F\big)^{\triangleleft}\in \TTT_K$ and $\eta(U)\subseteq \Ss_+$, Proposition \ref{p:basis-weak}(ii) implies that  $\eta$ is $\TTT_K$-continuous and hence  $\eta\in \widehat{(G,\TTT_K)}=\widehat{G}$. Since $\eta$ was arbitrary, we obtain  $\big( (n)K +F\big)^{\triangleleft\blacktriangleright}\subseteq \widehat{G}$. Thus $K\in \Mac(G)$ as desired.

Conversely, assume that $K\in \Mac(G)$ and let $\chi\in \widehat{(G,\TTT_K)}$. Then there are $n\in\NN$ and a finite  $0\in F\subseteq \widehat{G}$ such that $\chi\left( \big((n)K+F\big)^{\triangleleft}\right) \subseteq \Ss_+$. Hence $\chi\in \big((n)K+F\big)^{\triangleleft\blacktriangleright}$, and therefore $\chi\in \widehat{G}$. We proved that $\widehat{(G,\TTT_K)}\subseteq \widehat{G}$. Since $\sigma(G,\widehat{G})\leq \TTT_K$ we also have $\widehat{G}\subseteq \widehat{(G,\TTT_K)}$.  Thus $\widehat{G}= \widehat{(G,\TTT_K)}$ and hence $\TTT_K$ is compatible with $\tau$.\qed
\end{proof}


 Theorem \ref{t:Mackey-compatible-dual} implies that the family $\Mac(G)$ has some nice (hereditary) properties.
\begin{proposition} \label{p:Mac-her}
Let $(G,\tau)$ be a $MAP$ abelian group, and let $K$ be a quasi-convex subset of $\Gsa$. Then:
\begin{enumerate}
\item[{\rm(i)}] If $K\in\Mac(G)$, then $\qc_{\widehat{G}}(K_0)\in \Mac(G)$ for every subset $K_0$ of $K$.
\item[{\rm(ii)}] $K\in\Mac(G)$ if and only if $\qc_{\widehat{G}}\big( (n)K+F\big)\in\Mac(G)$ for all $n\in\NN$ and all finite $0\in F\subseteq \widehat{G}$.
\item[{\rm(iii)}] If $\TTT$ is a group topology on $G$ compatible with $\tau$, then $\{U^{\triangleright}: U\in\Nn(G,\TTT)\}\subseteq \Mac(G)$.
\item[{\rm(iv)}] If $A\subseteq \widehat{G}$ is equicontinuous $($for example, finite$)$, then $\qc_{\widehat{G}}(A)\in\Mac(G)$.
\item[{\rm(v)}] If $Y$ is a subgroup of $\widehat{G}$ and $Y\in\Mac(G)$, then $Y$ is weak${}^\ast$ compact, $Y^\top$ is $\TTT_Y$-open and $\widehat{G/Y^\top}=(G/Y^\top)^\ast$.
\item[{\rm(vi)}] Let $H$ be a quotient $MAP$ group of $G$, $p:G\to H$ be the quotient map, and $K\in\Mac(G)$. Set $\widetilde{K}:=K\cap p^\ast(\widehat{H})$ and $C:=(p^\ast)^{-1}(\widetilde{K})$. Then $C\in\Mac(H)$.
\end{enumerate}
\end{proposition}

\begin{proof}
(i) Set $K':=\qc_{\widehat{G}}(K_0)$. Being a closed subset of $K$, the quasi-convex set $K'$ is weak${}^\ast$ compact. Further, by Lemmas \ref{l:topology-G-dual}(iii) and \ref{l:topology-G-dual-weak}(i), we have  $\sigma(G,\widehat{G})\leq \TTT_{K'} \leq \TTT_K$. Since, by Theorem \ref{t:Mackey-compatible-dual}, the topology $\TTT_K$ is compatible with  $\tau$ we obtain that also $\TTT_{K'} $ is compatible with $\tau$. Applying once again  Theorem \ref{t:Mackey-compatible-dual} we obtain  $K'\in\Mac(G)$.

(ii) The sufficiency is clear. To prove the necessity let $K':=\qc_{\widehat{G}}\big( (n)K+F\big)$. Then, by Lemma \ref{l:topology-G-dual-weak}(ii), we have $\TTT_{K'}=\TTT_K$, and hence, by Theorem \ref{t:Mackey-compatible-dual}, $K'\in \Mac(G)$.

(iii) By Proposition \ref{p:polar-U-k-compact}, the set $K:=U^{\triangleright}$ is a $\sigma(\widehat{G},G)$-compact quasi-convex subset of $\widehat{G}$. To show that $K\in \Mac(G)$, fix an $n\in\NN$ and a finite subset $F=V^{\triangleright}$ of $\widehat{G}$, where $V$ is a quasi-convex $\sigma(G,\widehat{G})$-neighborhood of zero. We have to show that  every $\eta\in \big( (n)K +F\big)^{\triangleleft\blacktriangleright}$ belongs to $\widehat{G}$. To this end, we apply (iii) of Lemma \ref{l:polar-1} and obtain
\[
\eta\Big( \big( (n)K +(n)K\big)^{\triangleleft}\cap \big( F+F \big)^{\triangleleft}\Big) \subseteq \Ss_+.
\]
Choose a $\TTT$-neighborhood $W$  of zero such that $(2n)W\subseteq U\cap V\subseteq K^\triangleleft \cap F^\triangleleft$. Then, by (iv) of Lemma \ref{l:polar-1}, we have
\[
W\subseteq \big( (2n)K \big)^{\triangleleft} \cap \big( (2n)F\big)^{\triangleleft}\subseteq \big( (n)K +(n)K\big)^{\triangleleft} \cap \big( F +F\big)^{\triangleleft}
\]
and hence $\eta(W)\subseteq \Ss_+$. Therefore, by Proposition \ref{p:basis-weak}, $\eta$ is $\TTT$-continuous and hence $\eta\in \widehat{(G,\TTT)}=\widehat{G}$.  Thus $\big( (n)K +F\big)^{\triangleleft\blacktriangleright}\subseteq \widehat{G}$ and so $K\in \Mac(G)$, as desired.

(iv) Assume that $A$ is equicontinuous, i.e, there is $U\in\Nn_{s}(G,\tau)$ such that $A\subseteq U^{\triangleright}$. Then, by Proposition \ref{p:Mackey-equi}, also the set $K:=\qc_{\widehat{G}}(A)\subseteq U^{\triangleright}$ is equicontinuous.  Therefore $K^{\triangleleft}\in\Nn_{qc}(G,\tau)$, see Proposition \ref{p:Mackey-equi}. Thus, by (iii), $K=K^{\triangleleft\triangleright}$ belongs to $\Mac(G)$. 

(v) Since $Y\in\Mac(G)$, it is weak${}^\ast$ compact. By the definition of $\TTT_Y$, the inverse polar $Y^{\triangleleft}=Y^\top$ is a {\em $\TTT_Y$-open} subgroup of $G$. Denote by $p:G\to G/Y^\top$ the quotient map and let $\chi\in (G/Y^\top)^\ast$. Then $\chi\circ p\in (Y^\top)^{\blacktriangleright}= Y^{\triangleleft\blacktriangleright}$. Since $Y\in\Mac(G)$, Definition \ref{def:Mac(G)} (applied to $n=1$ and $F=\{0\}$) implies that $\chi\circ p\in \widehat{G}$ and hence $\chi\in \widehat{G/Y^\top}$.

(vi) By Lemma \ref{l:adjoint}, the adjoint map $p^\ast$ is an embedding of $\widehat{H}_{\sigma^\ast}$ onto a dually closed subgroup of $\Gsa$. Therefore $\widetilde{K}$ is a $\sigma(\widehat{G},G)$-compact quasi-convex subset of $\widehat{G}$, and hence $C$ is a $\sigma(\widehat{H},H)$-compact quasi-convex subset of $\widehat{H}$. To show that $C\in\Mac(H)$, fix $n\in\NN$, a finite $0\in F\subseteq \widehat{H}$ and let $\chi\in \big( (n)C +F\big)^{\triangleleft\blacktriangleright}$. Set $\eta:=p^\ast(\chi)\in G^\ast$ and $\widetilde{F}:=p^\ast(F)$.

For every $g\in \big( (n)\widetilde{K} +\widetilde{F}\big)^{\triangleleft}$ and each $\chi_1,\dots,\chi_n\in C$ and $f\in F$, we have $\eta_1:=p^\ast(\chi_1),\dots,\eta_n:=p^\ast(\chi_n)\in\widetilde{K}$ and $\widetilde{f}:=p^\ast(f)\in \widetilde{F}$, and hence
\[
\big(\chi_1+\cdots+\chi_n+f,p(g)\big)=\big(\eta_1+\cdots+\eta_n+\widetilde{f},g\big)\in\Ss_+,
\]
so $p(g)\in \big( (n)C +F\big)^{\triangleleft}$. Therefore $(\eta,g)=(\chi,p(g))\in \Ss_+$ for every $g\in \big( (n)\widetilde{K} +\widetilde{F}\big)^{\triangleleft}$, and hence $\eta\in \big( (n)\widetilde{K} +\widetilde{F}\big)^{\triangleleft\blacktriangleright}$. By (i), we have $\widetilde{K}\in\Mac(G)$. Consequently $\eta=p^\ast(\chi)=\chi\circ p\in \widehat{G}$ and hence $\chi\in \widehat{H}$. Thus $C\in\Mac(H)$.\qed
\end{proof}

\begin{proposition} \label{p:Mac-product}
Let $(G,\tau)$ and $(H,\nu)$  be $MAP$ abelian groups.  Then
\[
\Mac(G)\times \Mac(H) \subseteq \Mac(G\times H).
\]
\end{proposition}

\begin{proof}
Let $K\in \Mac(G)$ and $C\in \Mac(H)$. Clearly, $K\times C$ is a quasi-convex, compact subset of $\big(\widehat{G\times H}\big)_{\sigma^\ast}=\widehat{G}_{\sigma^\ast}\times \widehat{H}_{\sigma^\ast}$. Fix an $n\in\NN$ and a finite $0\in F\subseteq \widehat{G}\times \widehat{H} $. Without loss of generality we can assume that $F=F_G\times F_H$, where $0\in F_G\subseteq \widehat{G}$ and $0\in F_H\subseteq \widehat{H}$ are finite. Then
\[
F+ (n)K\times C = \big(F_G+ (n)K\big)\times \big(F_H+ (n)C\big).
\]
By (ii) of Lemma \ref{l:polar-1}, we have
\[
\big((2)F_G+ (2n)K\big)^\triangleleft +\big((2)F_G+ (2n)K\big)^\triangleleft \subseteq \big(F_G+ (n)K\big)^\triangleleft
\]
and, respectively,
\[
\big((2)F_H+ (2n)C\big)^\triangleleft +\big((2)F_H+ (2n)C\big)^\triangleleft \subseteq \big(F_H+ (n)C\big)^\triangleleft.
\]
Set $A:= \big((2)F_G+ (2n)K\big)^\triangleleft \times \big((2)F_H+ (2n)C\big)^\triangleleft$. Now, if $(g,h)\in A$, then for every $(\chi,\eta)\in F+ (n)K\times C$, (i) of Lemma \ref{l:polar-1} implies
\[
\big( (\chi,\eta), (g,h)\big) =(\chi,g)\cdot (\eta, h)\in \Ss_2 \cdot \Ss_2 =\Ss_+.
\]
Therefore $A\subseteq \big(F+ (n)K\times C\big)^\triangleleft$. Hence, for every $\xi=(\zeta,\phi)\in  \big(F+ (n)K\times C\big)^{\triangleleft\blacktriangleright} \subseteq G^\ast\times H^\ast$ and each $(g,h)\in A$, we have $\xi\in A^\blacktriangleright$.
Since $A\in \TTT_K\times\TTT_C$, Proposition \ref{p:basis-weak} implies that $\xi\in \big(G\times H, \TTT_K\times\TTT_C\big)^\wedge$. On the other hand, Theorem \ref{t:Mackey-compatible-dual} implies that $\TTT_K\times\TTT_C$ is compatible with $\tau\times\nu$. Therefore $\xi\in \widehat{G}\times \widehat{H}$. Thus $K\times C\in \Mac(G\times H)$.\qed
\end{proof}

The next theorem describes all compatible lqc group topologies on $(G,\tau)$ and solves Problem \ref{prob:description}.

\begin{theorem} \label{t:compatible-description}
Let $(G,\tau)$ be a $MAP$ abelian group, and let $\TTT$ be an lqc group topology on the abelian group $G$. Then $\TTT$ is compatible with $\tau$ if and only if there is a directed subset $\KK$ of $\Mac(G)$ such that $\TTT=\vee_{K\in \KK} \TTT_K$. In this case one can take $\KK=\{U^{\triangleright}: U\in\Nn_{qc}(G,\TTT)\}$.
\end{theorem}

\begin{proof}
Assume that $\TTT$ is compatible with $\tau$; in particular, $\sigma(G,\widehat{G})\leq\TTT$. Denote by $\KK$ the family of all subsets $K$ of $\widehat{G}$ such that $K=U^{\triangleright}$ for some quasi-convex $\TTT$-neighborhood $U$ of zero. Then, by Proposition \ref{p:Mac-her}(iii), $\KK\subseteq \Mac(G)$.

To show that $\KK$ is directed, let $K=U^{\triangleright},C=V^{\triangleright}\in \KK$. Set $W:=U\cap V$. Then $W$ is a quasi-convex $\TTT$-neighborhood of zero such that $K\cup C\subseteq W^{\triangleright}$. Thus $\KK$ is directed.

We claim that $\TTT=\vee_{K\in \KK} \TTT_K$. Indeed, set $\nu: =\vee_{K\in \KK} \TTT_K$. The construction of the family $\KK$ and the definition of $\TTT_K$ immediately imply $\TTT\leq\nu$. To prove the inverse inclusion $\TTT\geq\nu$, let $W$ be a $\nu$-neighborhood of zero. Then $W\in \TTT_{K_1}\vee\cdots\vee\TTT_{K_t}$ for some $K_1,\dots,K_t\in\KK$. Since $\KK$ is directed, there is $K\in\KK$ such that $\bigcup_{i=1}^t K_i \subseteq K$. Therefore, by (i) of Lemma \ref{l:topology-G-dual-weak}, $W\in \TTT_K$ and we can assume that $W=\big( (n)K+F\big)^{\triangleleft}$ for some $n\in\NN$, finite $F=V^{\triangleright}\subseteq \widehat{G}$ and $K=U^{\triangleright}\in\KK$, where $U\in\Nn_{qc}(G,\TTT)$ and $V$ is a quasi-convex $\sigma(G,\widehat{G})$-neighborhood of zero. Now choose $\mathcal{O}\in\Nn_{qc}(G,\TTT)$ such that $(2n)\mathcal{O}\subseteq U\cap V$. Then (iii) and (iv) of Lemma \ref{l:polar-1} imply
\[
\mathcal{O}\subseteq \big( (2n)K \big)^{\triangleleft} \cap \big( (2n)F\big)^{\triangleleft}\subseteq \big( (n)K +(n)K\big)^{\triangleleft} \cap \big( F +F\big)^{\triangleleft} \subseteq \big( (n)K+F\big)^{\triangleleft} =W.
\]
Thus $\nu\leq\TTT$ and hence $\TTT=\nu$.

Conversely, assume that $\KK$ is a directed subset of $\Mac(G)$ such that $\TTT: =\vee_{K\in \KK} \TTT_K$. We have to show that $\TTT$ is compatible with $\tau$. To this end, let $\chi\in \widehat{(G,\TTT)}$. Then there are $K_1,\dots, K_m\in \Mac(G)$, $n_1,\dots,n_m\in\NN$ and finite sets $F_1,\dots,F_m\subseteq \widehat{G}$ containing zero such that $\chi(V)\subseteq \Ss_+$, where
\[
V=\big( (n_1)K_1 +F_1\big)^{\triangleleft}\cap\cdots\cap \big( (n_m)K_m +F_m\big)^{\triangleleft}.
\]
Set $n:=n_1+\cdots+n_m$, $F:=F_1+\cdots+F_m$ and choose $K\in \KK$ such that $\bigcup_{i=1}^m K_i \subseteq K$. Then $W:=\big( (n)K +F\big)^{\triangleleft}\in \TTT_K$ and $W\subseteq V$. Clearly, $\chi(W)\subseteq \Ss_+$ and hence, by Proposition \ref{p:basis-weak}, $\chi$ is $\TTT_K$-continuous. As $\TTT_K$ is compatible with $\tau$ (see Theorem \ref{t:Mackey-compatible-dual}), we obtain $\chi\in \widehat{G}$. Thus $\TTT$ is compatible with $\tau$.\qed
\end{proof}

Proposition \ref{p:Mac-her} and Theorem \ref{t:compatible-description} motivate the next problem. Let $K_1,K_2\in\Mac(G)$. {\em When does $\qc_{\widehat{G}}(K_1\cup K_2)$ belong to $\Mac(G)$}? Below we answer this question.

\begin{proposition} \label{p:Mac(G)-nuclear}
Let $(G,\tau)$ be  a $MAP$ abelian group, and let $K_1,K_2\in \Mac(G)$. Then  the following assertions are equivalent:
\begin{enumerate}
\item[{\rm(i)}] there is $K\in \Mac(G)$ such that $K_1\cup K_2\subseteq K$;
\item[{\rm(ii)}] $\qc_{\widehat{G}}(K_1\cup K_2)\in \Mac(G)$;
\item[{\rm(iii)}] the topology $\TTT_{K_1}\vee \TTT_{K_2}$ is compatible with $\tau$;
\item[{\rm(iv)}] the diagonal subgroup $\Delta:=\{(g,g): g\in G\}$ of the product $L:=(G,\TTT_{K_1})\times (G,\TTT_{K_2})$ is dually embedded.
\end{enumerate}
\end{proposition}

\begin{proof}
(i)$\Rightarrow$(ii)  By assumption, there is $K\in\Mac(G)$ such that $K_1\cup K_2\subseteq K$. Therefore, by (i) of Proposition \ref{p:Mac-her}, $\qc_{\widehat{G}}(K_1\cup K_2)\in \Mac(G)$ as desired.

(ii)$\Rightarrow$(iii) Set $K':=\qc_{\widehat{G}}(K_1\cup K_2)$, so $K'\in \Mac(G)$. Then,  by Lemmas \ref{l:topology-G-dual}(iii) and \ref{l:topology-G-dual-weak}(i), we have  $\sigma(G,\widehat{G})\leq \TTT_{K_1}\vee \TTT_{K_2}\leq \TTT_{K'}$. Since, by Theorem \ref{t:Mackey-compatible-dual}, the topology $\TTT_{K'}$ is compatible with  $\tau$ so must be also $\TTT_{K_1}\vee \TTT_{K_2}$.

(iii)$\Rightarrow$(i) By assumption the topology $\nu:=\TTT_{K_1}\vee \TTT_{K_2}$ is compatible with $\tau$. Therefore, by Theorem \ref{t:compatible-description}, $\nu=\vee_{K\in \KK} \TTT_K$ where $\KK=\{U^{\triangleright}: U\in\Nn_{qc}(G,\nu)\}$ is a directed subfamily of $\Mac(G)$. Since, for $i\in\{1,2\}$,  $\TTT_{K_i}\leq \nu$ and $K_i^{\triangleleft}$ is a quasi-convex $\TTT_{K_i}$-neighborhood of zero, we obtain that $K_i^{\triangleleft}\in\Nn_{qc}(G,\nu)$ and hence $K_i=K_i^{\triangleleft\triangleright}\in \KK$. Since $\KK$ is directed, there is $K\in\KK\subseteq \Mac(G)$ such that $K_1\cup K_2\subseteq K$.

(iii)$\Rightarrow$(iv) Let $\chi\in \widehat{\Delta}$. Since $\delta:(G, \TTT_{K_1}\vee \TTT_{K_2})\to \Delta$, $\delta(g):=(g,g)$,  is a topological isomorphism, we obtain $\delta^\ast(\chi)\in \widehat{(G, \TTT_{K_1}\vee \TTT_{K_2})}=\widehat{G}$. Consider $\eta:=(\delta^\ast(\chi),0)\in \widehat{L}$. Then for every $g\in G$, we have
\[
\big(\eta,(g,g)\big)=(\delta^\ast(\chi),g)=(\chi,\delta(g))=\big(\chi,(g,g)\big)
\]
and hence $\eta$ extends $\chi$. Thus $\Delta$ is dually embedded in $L$.

(iv)$\Rightarrow$(iii) Our proof follows the proof of Theorem 2.7 of \cite{Gab-Mackey}. 
By assumption,  for every $\widetilde{\chi}\in \widehat{\Delta}$ there is $(\xi,\zeta)\in \widehat{L}$ such that $\big( \widetilde{\chi}, (g,g)\big)=(\xi,g)(\zeta,g)$ for all $g\in  G$. As $K_1,K_2\in \Mac(G)$, Theorem \ref{t:Mackey-compatible-dual} implies $\xi,\zeta\in \widehat{G}$ and hence $\big( \widetilde{\chi}, (g,g)\big)=(\xi+\zeta,g)$.

Set $H:=(G,\TTT_{K_1}\vee \TTT_{K_2})$ and observe that the canonical map $\delta: H\to L$, $\delta(g):=(g,g)$ for $g\in G$, is an embedding of $H$ onto $\Delta$.
Now, let $\chi\in \widehat{H}$, fix ${\bar\chi}\in(\delta^\ast)^{-1}(\chi)\in \widehat{L}$ (such ${\bar\chi}$ exists because $\Delta$ is dually embedded in $L$) and set $\widetilde{\chi}:={\bar\chi}{\restriction}_\Delta \in \widehat{\Delta}$. Then
\[
(\chi,g)=\big(\delta^\ast ({\bar\chi}),g\big)=({\bar\chi},\delta(g))=(\widetilde{\chi},\delta(g))=(\xi+\zeta,g) \quad (g\in G).
\]
Therefore $\chi=\xi+\zeta\in \widehat{G}$ and hence $\widehat{H}\subseteq \widehat{G}$. Since $\sigma(G,\widehat{G})\leq \TTT_{K_1}\leq \TTT_{K_1}\vee \TTT_{K_2}$, it is clear that $\widehat{G}\subseteq \widehat{H}$. Thus $\widehat{H}=\widehat{G}$ and hence $\TTT_{K_1}\vee \TTT_{K_2}$ is compatible with $\tau$.\qed
\end{proof}


The following characterization of lqc abelian groups which have a Mackey group topology is more precise than Theorem \ref{t:Char-Mackey-MP}.
\begin{theorem} \label{t:Mackey-exist}
For  a locally quasi-convex abelian group $(G,\tau)$ the following assertions are equivalent:
\begin{enumerate}
\item[{\rm(i)}] $G$ has a Mackey group topology;
\item[{\rm(ii)}]  the family $\Mac(G)$ is directed;
\item[{\rm(iii)}] $\qc_{\widehat{G}}(K_1\cup K_2)\in \Mac(G)$ for every $K_1,K_2\in \Mac(G)$;
\item[{\rm(iv)}]  for every $K_1,K_2\in \Mac(G)$, the topology $\TTT_{K_1}\vee \TTT_{K_2}$ is compatible with $\tau$;
\item[{\rm(v)}] for every $K_1,K_2\in \Mac(G)$, the diagonal group $\Delta=\{(g,g): g\in G\}$ is dually embedded in the product $(G,\TTT_{K_1})\times (G,\TTT_{K_2})$.
\end{enumerate}
If {\rm (i)-(v)} are satisfied, the topology $\mu: =\vee_{K\in \Mac(G)} \TTT_K$ is Mackey.
\end{theorem}

\begin{proof}
(i)$\Rightarrow$(ii) Assume that $G$ has a Mackey group topology $\mu$. Let $K_1,K_2 \in \Mac(G)$. Then, by Theorem \ref{t:Mackey-compatible-dual}, $\TTT_{K_1}$ and $\TTT_{K_2}$ are compatible with $\mu$. Since $\mu$ is Mackey,  there are $U_1, U_2\in \mu$ such that $U_1\subseteq K_1^{\triangleleft}$ and $U_2\subseteq K_2^{\triangleleft}$. Set $U:=U_1\cap U_2$ and $K:=U^{\triangleright}\subseteq \widehat{G}$. Then $K_1\cup K_2 \subseteq K$ and, by (ii) of Lemma \ref{l:topology-G-dual}, $\TTT^{\mu}_K=\mu$. Therefore, by (iii) of Lemma \ref{l:topology-G-dual}, $\sigma(G,\widehat{G})\leq \TTT_K \leq \TTT^{\mu}_K=\mu$, and hence $\TTT_K$ is compatible with $\mu$. Now Theorem \ref{t:Mackey-compatible-dual} implies  $K\in \Mac(G)$. Thus $\Mac(G)$ is directed.

(ii)$\Rightarrow$(i) Let $\Mac(G)$ be directed.  Then, by Theorem \ref{t:compatible-description}, the topology $\mu: =\vee_{K\in \Mac(G)} \TTT_K$ is compatible with $\tau$ and any compatible topology (with $\tau$) is weaker than $\mu$. This means that $\mu$ is a Mackey group topology on $G$.

The equivalences (ii)$\Leftrightarrow$(iii)$\Leftrightarrow$(iv)$\Leftrightarrow$(v) follow from (i) of Proposition \ref{p:Mac-her} and Proposition \ref{p:Mac(G)-nuclear}.\qed
%
%
\end{proof}

In the next theorem we characterize Mackey groups.
\begin{theorem} \label{t:Mackey-group-equi}
An lqc abelian group $(G,\tau)$ is a Mackey group if and only if every $K\in\Mac(G)$ is equicontinuous. In this case $\tau =\vee_{K\in \Mac(G)} \TTT_K$.
\end{theorem}

\begin{proof}
Assume that $(G,\tau)$ is a Mackey group. Then, by Theorem \ref{t:Mackey-exist}, $\tau=\vee_{K\in \Mac(G)} \TTT_K$. Therefore, for every $K\in \Mac(G)$, the inverse polar $K^{\triangleleft}$ is a neighborhood of zero in $\tau$. Thus, by Proposition \ref{p:Mackey-equi}, $K=K^{\triangleleft\triangleright}$ is equicontinuous.

Conversely, assume that every $K\in \Mac(G)$ is equicontinuous and let $\TTT$ be an lqc group topology on $G$ compatible with $\tau$. Then, by (iii) of Proposition \ref{p:Mac-her}, for every quasi-convex $\TTT$-neighborhood $U$ of zero, we have $U^{\triangleright}\in \Mac(G)$ and hence $U^{\triangleright}$ is equicontinuous. Then Proposition \ref{p:Mackey-equi} implies that $U=U^{\triangleright\triangleleft}$ is a $\tau$-neighborhood of zero. Thus $\TTT\leq \tau$ and hence $\tau$ is a Mackey group topology. \qed
\end{proof}


Following \cite{CMPT} (resp. \cite{MT}), a $MAP$ abelian group $G$ is called {\em $g$-barrelled} ({\em $c_0$-barrelled}) if any $\sigma(\widehat{G},G)$-compact subset (resp. $\sigma(\widehat{G},G)$-convergent sequence)  of  $\widehat{G}$ is equicontinuous. As an immediate corollary of Theorem \ref{t:Mackey-group-equi} we obtain the next assertion proved in \cite[Theorem~4.2(1)]{CMPT}.
\begin{corollary}[\cite{CMPT}] \label{c:g-barrelled-Mackey}
Every $g$-barrelled lqc abelian group is a Mackey group.
\end{corollary}

\begin{remark}
We noticed in Remark 6.4 of \cite{Gab-Respected} that there are Mackey groups $G$ which are not  $c_0$-barrelled. Therefore the family $\Mac(G)$ is smaller than the family of all $\sigma(\widehat{G},G)$-compact subsets of $\widehat{G}$.\qed
\end{remark}

Let us note the following interesting result.
\begin{proposition}[\cite{BTAVM}] \label{p:bounded-Mackey-precom}
Let $(G,\tau)$ be an lqc abelian group of finite exponent such that $|\widehat{G}|<\mathfrak{c}$. Then $G$ is a Mackey group.
\end{proposition}

\begin{proof}
By Theorem \ref{t:Mackey-group-equi}, we have to show that every $K\in\Mac(G)$ is equicontinuous. So, fix $K\in \Mac(G)$ and a natural number $n> \exp(G)$.

We claim that $\big( (n)K\big)^{\triangleleft}=\langle K\rangle^\top$. Indeed, let $g\in \big( (n)K\big)^{\triangleleft}$. Since $n> \exp(G)$ and $0\in K$,  we obtain
\[
\big( i\chi,g\big)=\big( \chi,g\big)^i\in \Ss_+ \; \mbox{ for all $\chi\in K$ and $i=1,\dots,o(g)$},
\]
that is possible if and only if $(\chi,g)=1$. Thus $g\in\langle K\rangle^\top$ and hence   $\big( (n)K\big)^{\triangleleft}\subseteq \langle K\rangle^\top $. The inverse inclusion holds trivially. The claim is proved.

By Proposition \ref{p:Mac-her}(ii), $\qc_{\widehat{G}}\big( (n)K\big) =\big( (n)K\big)^{\triangleleft\triangleright}\in\Mac(G)$. Therefore, by the claim, the set $Y:=\big( (n)K\big)^{\triangleleft\triangleright}=\langle K\rangle^{\top\perp}$ is a weak${}^\ast$ compact subgroup of $\widehat{G}$.
It is well known that every infinite compact group has size $\geq\mathfrak{c}$ and since $|\widehat{G}|<\mathfrak{c}$, we obtain that $\langle K\rangle^{\top\perp}$ and hence also $K$ are finite, and therefore $K$ is equicontinuous.\qed
\end{proof}

The next theorem gives an answer to Problem \ref{prob:product-Mackey} for the case of two groups.
\begin{theorem} \label{t:Mackey-product}
For lqc abelian groups $G$ and $H$ the following assertions are equivalent:
\begin{enumerate}
\item[{\rm(i)}] the product $G\times H$ is a Mackey group;
\item[{\rm(ii)}] for every $S\in \Mac(G\times H)$, the projections $K:=\pi_{\widehat{G}}(S)$ and $C:=\pi_{\widehat{H}}(S)$ of $S$ onto $\widehat{G}$ and $\widehat{H}$ are equicontinuous;
\item[{\rm(iii)}] $G$ and $H$ are Mackey groups and $\Mac(G)\times \Mac(H)$ swallows $\Mac(G\times H)$.
\end{enumerate}
\end{theorem}

\begin{proof}
(i)$\Rightarrow$(ii) Let $S\in \Mac(G\times H)$. By Theorem \ref{t:Mackey-group-equi}, $S$ is equicontinuous. Take  $U\times V\in\Nn(G\times H)=\Nn(G)\times \Nn(H)$ such that $(\chi,(g,h))\in \Ss_+$ for every $(g,h)\in U\times V$ and each $\chi=(\eta,\xi)\in S$. Now, for every $g\in U$ and each $\eta\in K$, choose $\xi\in C$ such that $\chi:=(\eta,\xi)\in S$ and then
\[
(\eta,g)=(\eta,g)\cdot(\xi,0)=(\chi,(g,0))\in \Ss_+.
\]
Therefore, by Proposition \ref{p:Mackey-equi}, $K$ is equicontinuous. Analogously one can prove that $C$ is equicontinuous as well.



(ii)$\Rightarrow$(iii) To show that $G$ is a Mackey group, let $K\in \Mac(G)$. Then, by Proposition \ref{p:Mac-product}, $K\times\{0\}\in \Mac(G\times H)$. So, by assumption, $K=\pi_{\widehat{G}}(K\times\{0\})$ is equicontinuous. Thus, by Theorem \ref{t:Mackey-group-equi}, the group $G$ is Mackey. Analogously, one can prove that $H$ is a Mackey group.

To show that $\Mac(G)\times \Mac(H)$ swallows $\Mac(G\times H)$, fix $S\in \Mac(G\times H)$. Then, by (ii), the projections $K:=\pi_{\widehat{G}}(S)$ and $C:=\pi_{\widehat{H}}(S)$ of $S$ onto $\widehat{G}$ and $\widehat{H}$ are equicontinuous. Therefore, by (iv) of Proposition \ref{p:Mac-her}, $\qc_{\widehat{G}}(K)\in\Mac(G)$ and $\qc_{\widehat{H}}(C)\in\Mac(H)$. It remains to note that $S\subseteq \qc_{\widehat{G}}(K)\times \qc_{\widehat{H}}(C)$.

(iii)$\Rightarrow$(i) Let  $S\in \Mac(G\times H)$. Take $K\in\Mac(G)$ and $C\in\Mac(H)$ such that $S\subseteq K\times C$. Since $G$ and $H$ are Mackey groups, $K$ and $C$ are equicontinuous by Theorem \ref{t:Mackey-group-equi}. Therefore $K\times C$ and hence also $S$ are equicontinuous. Once more applying Theorem \ref{t:Mackey-group-equi} we obtain that $G\times H$ is a Mackey group.\qed
\end{proof}
We do not know whether the condition  in Theorem \ref{t:Mackey-product} that $\Mac(G)\times \Mac(H)$ swallows $\Mac(G\times H)$ can be omitted, in other words whether the following problem firstly selected by E.~Mart\'{\i}n Peinador has an affirmative answer:
\begin{problem} \label{prob:product-Mackey-MP}
Is the product of two Mackey groups Mackey?
\end{problem}

It is well known \cite[Proposition~8.5.8]{Jar} that the completion of a Mackey lcs is a Mackey space.
\begin{problem}[\cite{AD-mackey}] \label{prob:completion-Mackey}
Is it true that the completion of a Mackey group is a Mackey group?
\end{problem}

Although we do not know an answer to Problem \ref{prob:completion-Mackey}, nevertheless we are able to clarify a little this problem in the next assertion.
\begin{proposition} \label{p:Mackey-subgroup-dense}
Let $H$ be a dense subgroup of an lqc abelian group $(G,\tau)$ such that $(H,\tau{\restriction}_H)$ is a Mackey group. Then $G$ is a Mackey group if and only if $H$ is dense in $(G,\nu)$ for every group topology $\nu\in \mathcal{C}(G,\tau)$.
\end{proposition}

\begin{proof}
If $G$ is a Mackey group and $\nu\in \mathcal{C}(G,\tau)$, then $\nu\leq \tau$ and hence $H$ is also dense in $(G,\nu)$. Conversely, assume that $H$ is dense in $(G,\nu)$ for every $\nu\in \mathcal{C}(G,\tau)$. Then
\[
\widehat{(H,\nu{\restriction}_H})=\widehat{(G,\nu)}=\widehat{G}=\widehat{(H,\tau{\restriction}_H)},
\]
and hence $\nu{\restriction}_H$ is compatible with $\tau{\restriction}_H$. As $(H,\tau{\restriction}_H)$ is Mackey we obtain $\nu{\restriction}_H\leq\tau{\restriction}_H$. Let $W\subseteq G$ be a $\nu$-neighborhood of zero. Since $(G,\nu)$ is locally quasi-convex, we can choose a $\nu$-open neighborhood $V$ of zero in $G$  such that $V^{\triangleright\triangleleft}\subseteq W$. Then $V\cap H\in \nu{\restriction}_H\leq\tau{\restriction}_H$. Observe also that since $V\cap H$ is $\nu$-dense in $V$, we have $\big(V\cap H)^{\triangleright}=V^\triangleright$. Therefore
$
\overline{V\cap H}^{\,\tau} \subseteq \big(V\cap H)^{\triangleright\triangleleft} =  V^{\triangleright\triangleleft}\subseteq W.
$
Since $\overline{V\cap H}^{\,\tau}$ is a $\tau$-neighborhood of zero in $G$, it follows that $\nu\leq \tau$. Thus $\tau$ is a Mackey topology on $G$.\qed
\end{proof}

We finish this section by some interesting application of the obtained results.
If $(G,\tau)$ is a nuclear group (for the definition of nuclear groups we refer the reader to \cite{Ban}), we shall say that the topology $\tau$ is {\em nuclear}. Denote by $\Nuc$ the class of all nuclear groups.

\begin{proposition}[\cite{Ban}] \label{p:Nuclear-prop}
{\rm(i)}  Every LCA group is nuclear.

{\rm(ii)} Each subgroup of a nuclear group $G$ is dually embedded in $G$.

{\rm(iii)} $\Nuc$ is closed under taking products, Hausdorff quotients and subgroups.
\end{proposition}

Recall that a group topology $\tau$ on an abelian group $G$ is called {\em linear} if it admits a neighborhood basis at zero consisting of subgroups.
It is known (see \cite{AG}) that every linear abelian group embeds into the product of a family of discrete abelian groups, and hence every linear abelian group is nuclear.

The next theorem is one of the main results of \cite{Aus-Dikr-boun}.

\begin{theorem}[\cite{Aus-Dikr-boun}] \label{t:Mackey-finite-expon}
Every lqc abelian group $(G,\tau)$ of finite exponent has a Mackey group topology.
\end{theorem}

\begin{proof}
By Theorem \ref{t:Mackey-exist}, it suffices to show that for every $K_1,K_2\in \Mac(G)$, the topology $\TTT_{K_1}\vee \TTT_{K_2}$ is compatible with $\tau$.
Since $\TTT_{K_1}$ and $\TTT_{K_2}$ are lqc group topologies and $G$ is of finite exponent,  Proposition 2.1 of \cite{AG} (which states that every lqc group topology on an abelian group of finite exponent is linear) implies that  the groups $(G,\TTT_{K_1})$ and $(G,\TTT_{K_2})$ are linear and hence they are nuclear groups. Thus, by Propositions \ref{p:Mac(G)-nuclear} and \ref{p:Nuclear-prop}, $\TTT_{K_1}\vee \TTT_{K_2}$ is compatible with $\tau$. \qed
\end{proof}
If $G$ is an infinite discrete abelian group,
we denote by $G^+$ the group $G$  endowed with the Bohr topology. Then $G^+$ is a precompact group and hence it is locally quasi-convex. Clearly, $G^+$ is not a Mackey group. Therefore, if additionally $G$ has finite exponent, the conclusion of Theorem \ref{t:Mackey-finite-expon} cannot be strengthen to ``is a Mackey group''. However, in the important case when the bounded group $(G,\tau)$ is reflexive, it is a Mackey group by Theorem 6.7 of \cite{Gab-Respected}.


\section{Mackey group problem for locally convex spaces} \label{sec:Mackey-group-lcs}


The most important and widely studied classes of locally convex spaces and relations between them are given in the following diagram (for definitions and proofs see the books \cite{Jar,kak})
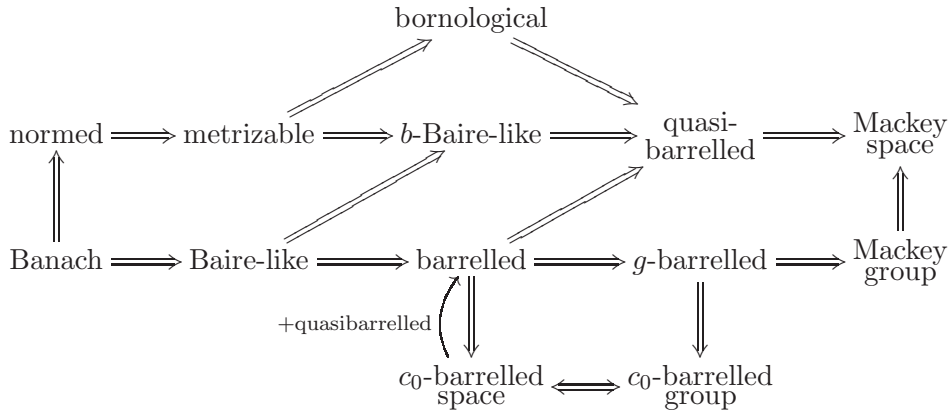
\[
\xymatrix{
& & \mbox{bornological}  \ar@{=>}[rd] & & \\
\mbox{normed}  \ar@{=>}[r] & \mbox{metrizable}  \ar@{=>}[r] \ar@{=>}[ru] & \mbox{$b$-Baire-like}  \ar@{=>}[r] & {\substack{\mbox{quasi-} \\
\mbox{barrelled}}} \ar@{=>}[r] & {\substack{\mbox{Mackey} \\ \mbox{space}}}\\
\mbox{Banach}  \ar@{=>}[r] \ar@{=>}[u] & \mbox{Baire-like}  \ar@{=>}[r] \ar@{=>}[ru] & \mbox{barrelled} \ar@{=>}[ru]\ar@{=>}[d] \ar@{=>}[r] & \mbox{$g$-barrelled}\ar@{=>}[r]\ar@{=>}[d]  & {\substack{\mbox{Mackey} \\ \mbox{group}}} \ar@{=>}[u] \\
& & {\substack{\mbox{$c_0$-barrelled} \\ \mbox{space}}} \ar@{<=>}[r] \ar@/^1pc/[u]^{\mbox{\scriptsize +quasibarrelled}} & {\substack{\mbox{$c_0$-barrelled} \\ \mbox{group}}} &
}
\]
and none of these implications is reversible. Below we note and discuss only some related results concerning the $g$-barrelledness and the property of being a Mackey group.


\begin{remark} \label{rem:lcs-Mackey}
(i) In Remark~16 of \cite{CMPT}, it is stated that for a non-reflexive real Banach space  $E$, the space $\big(E',\mu(E',E)\big)$ is a $g$-barrelled lcs which  is not barrelled (where $\mu(E',E)$ is the Mackey topology on $E'$).
A similar example  with a detailed proof of a $g$-barrelled but not barrelled lcs is given in Example 5.6 
of \cite{Gab-Respected}.

(ii) The first example of a metrizable lcs $E$ which is not a Mackey group is given in Theorem 3.1 of \cite{Gab-Mackey}. Moreover, there are even normed spaces which are not Mackey groups, see \cite{Gabr-normed-Mackey}.

(iii) For a Tychonoff space $X$ we denote by $C_p(X)$ the space $C(X)$ of all real-valued continuous functions endowed with the pointwise topology. 
Although $C_p(X)$  being quasibarrelled is always a Mackey space, the main result of \cite{Gab-Cp} states that the $C_p(X)$ is a Mackey group if and only if it is barrelled.

(iv) If an lcs $E$ is a Mackey group, then it is a Mackey space by Lemma 2.3 
of \cite{Gabr-L(X)-Mackey}.

(v) If $E$ is an lcs, then $E$ is a $c_0$-barrelled space if and only if it is a $c_0$-barrelled group by Proposition 5.1(i) of \cite{Gab-Respected}. If additionally $E$ is a quasibarrelled space, then $E$  is $c_0$-barrelled if and only if $E$ is a barrelled space by Proposition 12.2.3 of \cite{Jar}.\qed
%
\end{remark}



The diagram and Remark \ref{rem:lcs-Mackey} suggest the following natural problem.
\begin{problem} \label{prob:normed-Mackey-group}
Does there exist a non-quasibarrelled $c_0$-barrelled space $E$ which is a Mackey group?
\end{problem}

 Let $\mathcal{E}$ be a class of Mackey locally convex spaces (for example,  $\mathcal{E}$ is a class from the first, second or forth lines in the above diagram).
Taking into account (ii) and (iv) of Remark \ref{rem:lcs-Mackey}, one can naturally to consider the next question which is a partial but more concrete case of Problem \ref{prob:lcs-Mackey-group}.
\begin{problem}\label{prob:lcs-Mackey-group-E}
Characterize those spaces $E\in \mathcal{E}$ which are Mackey groups.
\end{problem}

Below we answer  Problem \ref{prob:lcs-Mackey-group} and Problem \ref{prob:lcs-Mackey-group-E} in terms of the family $\Mac(E)$ and the family  $\mathcal{CAC}(E')$  of all $\sigma(E',E)$-compact absolutely convex subsets of $E'$. We start from the following relation between these families. Recall that $\psi:E'\to \widehat{E}, \psi(\chi):= e^{2\pi i \chi}$, denotes the canonical isomorphism of $E'$ onto $\widehat{E}$.
\begin{proposition} \label{p:lcs-Mac(E)}
Let $(E,\tau)$ be an locally convex space, and let $K\in \mathcal{CAC}(E')$. Then $\psi(K)\in\Mac(E)$.
\end{proposition}

\begin{proof}
Proposition 2.11 of \cite{Gab-Respected} states that 
a subset $C$ of $E'$ is $\sigma(E',E)$-compact if and only if $\psi(C)$ is a $\sigma(\widehat{E},E)$-compact subset of $\widehat{E}$. Therefore, $\psi(K)$ is $\sigma(\widehat{E},E)$-compact. As $K$ is absolutely convex and closed, the set $\psi(K)$ is  quasi-convex by Lemma \ref{l:polar-neigh}.

To check that $\psi(K)\in\Mac(E)$, fix $n\in\NN$ and a finite $0\in F\subseteq \widehat{E}=\psi(E')$. We have to show that $\big( (n)\psi(K) +F\big)^{\triangleleft\blacktriangleright} \subseteq \widehat{E}$.
Observe that $(n)K=nK$ is absolutely convex, and hence, by Theorem 4.4.4 of \cite{NaB}, the absolutely convex hull $\widetilde{K}$ of $(n)K+\psi^{-1}(F)=\bigcup_{f\in \psi^{-1}(F)} (f+nK)$ is also $\sigma(E',E)$-compact. So $\widetilde{K}\in \mathcal{CAC}(E')$ and $\psi(\widetilde{K})$ is quasi-convex by Lemma \ref{l:polar-neigh}. Therefore, replacing $(n)\psi(K)+F$ by the bigger set $\psi(\widetilde{K})$ if needed,  we can assume that $n=1$ and $F=\{0\}$.

To show that $\psi(K)^{\triangleleft\blacktriangleright} \subseteq \widehat{E}$, fix an arbitrary $\chi\in \psi(K)^{\triangleleft\blacktriangleright}$.
By Lemma \ref{l:polar-neigh}, $\psi(K)^{\triangleleft}=\tfrac{1}{4} K^\circ$ is an absolutely convex, closed subset of $E$. By Theorem 8.8.4 of \cite{NaB},  $\psi(K)^{\triangleleft}$ is a barrel in $E$ and hence it is absorbent. Note that in the additive representation of $\Ss$ as the quotient group $\TT:=\IR/\ZZ$, the inclusion $\chi(A)\subseteq \Ss_+$ means that $\chi(A)\in \big[-\tfrac{1}{4},\tfrac{1}{4}\big]$. Therefore we can apply Lemma 2.2 of \cite{Ban} to find a (unique) linear functional $\eta$ on $E$ such that $\psi(\eta)=\chi$ and
\[
(\eta, x)\in \big[-\tfrac{1}{4};\tfrac{1}{4}\big] \;\; \mbox{ for every }\; x\in \tfrac{1}{4} K^\circ,
\]
i.e. $\eta(K^\circ)\subseteq [-1,1]$. Recall that, by the Mackey--Arens theorem, $K^\circ$ is a neighborhood of zero in the Mackey space topology $\mu$ of $E$. Therefore, $\eta$ is $\mu$-continuous. But since $\mu$ is compatible with $\tau$, it follows that $\eta$ is continuous. Thus $\chi=\psi(\eta)\in \widehat{E}$ and hence $\psi(K)^{\triangleleft\blacktriangleright} \subseteq \widehat{E}$, as desired. \qed
\end{proof}


\begin{theorem} \label{t:Mackey-space-group}
A  locally convex space  $(E,\mu)$ is a Mackey group if and only if it is a Mackey space and $\psi\big(\mathcal{CAC}(E')\big)$ swallows $\Mac(E)$.
\end{theorem}

\begin{proof}
Assume that  $(E,\mu)$ is a Mackey group. Then, by Lemma 2.3 of \cite{Gabr-L(X)-Mackey},  
$E$ is a Mackey space and, by Theorem \ref{t:Mackey-group-equi}, every $C\in\Mac(G)$ is equicontinuous. Therefore, by (ii) of Theorem \ref{t:Mackey-Arens}, there is $K\in \mathcal{CAC}(E')$ such that $C\subseteq (K^\circ)^\triangleright$. By Lemma \ref{l:polar-neigh}, we have $(K^\circ)^\triangleright =\psi(\tfrac{1}{4}K)$. Thus $\psi\big(\mathcal{CAC}(E')\big)$ swallows $\Mac(E)$.

Conversely, let  $(E,\mu)$ be a Mackey space and $\psi\big(\mathcal{CAC}(E')\big)$ swallow $\Mac(E)$. Since the family $\mathcal{CAC}(E')$ is directed, we obtain that the family $\Mac(E)$ is directed as well. Therefore, by Theorem \ref{t:Mackey-exist}, $E$ has a Mackey group topology $\TTT$. Since $\mu\leq\TTT$ holds always, we shall prove that $\mu\geq\TTT$. To this end, fix a $\TTT$-quasi-convex neighborhood $U$ of zero in $E$. By (iii) of Proposition \ref{p:Mac-her}, we can assume that $U=C^\triangleleft$ for some $C\in\Mac(E)$. By assumption, there is $K\in\mathcal{CAC}(E')$ such that $C\subseteq \psi(K)$. Then, by Lemma \ref{l:polar-neigh}, we have $\psi(K)^\triangleleft =\tfrac{1}{4} K^\circ \subseteq C^\triangleleft=U$. Since $\mu$ is a Mackey space topology, we have $\tfrac{1}{4} K^\circ \in\mu$. Therefore, $U$ is a $\mu$-neighborhood of zero. Thus $\TTT\leq \mu$ and hence $\TTT=\mu$.\qed
\end{proof}

\begin{corollary} \label{c:quasibar-Mackey-group}
A Mackey (in particular, quasibarrelled) locally convex space $E$ is a Mackey group if and only if $\psi\big(\mathcal{CAC}(E')\big)$ swallows $\Mac(E)$.
\end{corollary}

\begin{remark} \label{rem:Mac-acx-compact}
(i) There are metrizable spaces $E$ for which $\psi\big(\mathcal{CAC}(E')\big)$ does not swallow $\Mac(E)$. Indeed, by (iii) if Remark \ref{rem:lcs-Mackey}, the space $C_p(\mathbb{Q})$ is not a Mackey group. Thus, by Corollary \ref{c:quasibar-Mackey-group}, there is $K\in\Mac(C_p(\mathbb{Q}))$ which cannot be covered by any element of $\psi\big(\mathcal{CAC}(E')\big)$.

(ii) The sufficiency in Corollary \ref{c:quasibar-Mackey-group} can be proved in another way: If $\psi\big(\mathcal{CAC}(E')\big)$ swallows $\Mac(E)$, then every $K\in\Mac(E)$ is contained in some $\psi(C)$ with $C\in \mathcal{CAC}(E')$. Since $E$ is quasibarrelled, $C$ is equicontinuous by Theorems 11.11.5 and 11.11.4 of \cite{NaB}. Therefore, by (vi) of Proposition 2.11 of \cite{Gab-Respected}, $\psi(C)$ and hence also $K$ are also equicontinous. Finally, by Theorem \ref{t:Mackey-group-equi}, $E$ is a Mackey group.\qed 
\end{remark}

\begin{corollary} \label{c:normed-Mackey-group}
A normed space $E$ is a Mackey group if and only if for every $C\in\Mac(E)$, the compact set $\psi^{-1}(C)$ is norm bounded in the dual Banach space $E'$.
\end{corollary}

\begin{proof}
If $E$ is a Mackey group, then, by Corollary \ref{c:quasibar-Mackey-group}, there is $K\in \mathcal{CAC}(E')$ such that $\psi^{-1}(C)\subseteq K$. Since $E$ is quasibarrelled, by Theorem 11.11.5 of \cite{NaB}, $K$ is a bounded subset of the strong dual  space $E'_\beta$. It remains to note that  $E'_\beta$ is the Banach dual $E'$.

Conversely, if   for every $C\in\Mac(E)$, the compact set $\psi^{-1}(C)$ is bounded in the dual Banach space $E'$. Then $\psi^{-1}(C)$ is equicontinuous. Therefore, by (vi) of Proposition 2.11 of \cite{Gab-Respected}, $C$ is also equicontinous. Now Theorem \ref{t:Mackey-group-equi} applies to get that $E$ is a Mackey group.\qed
\end{proof}

\begin{corollary} \label{c:vector-Mackey-group}
An lcs $E$ admits a Mackey group topology $\mu$ which is a vector topology if and only if $\psi\big(\mathcal{CAC}(E')\big)$ swallows $\Mac(E)$. In this case $\mu$ is a Mackey space topology.
\end{corollary}

\begin{proof}
If $(E,\mu)$ is also a topological  vector space, then, by Proposition 2.4 of \cite{Ban}, $(E,\mu)$ is a locally convex space. By Theorem \ref{t:Mackey-space-group}, $\psi\big(\mathcal{CAC}(E')\big)$ swallows $\Mac(E)$. Conversely, let $\psi\big(\mathcal{CAC}(E')\big)$ swallow $\Mac(E)$ and let $\mu$ be a Mackey space topology on $E$. Then, by Theorem \ref{t:Mackey-space-group}, $\mu$ is a Mackey group topology. \qed
\end{proof}

By Remark \ref{rem:Mac-acx-compact} and Corollary \ref{c:vector-Mackey-group}, if a Mackey group topology exists for $E$ it is usually not a vector topology.
It would be excellent to find another characterization (than in Theorem \ref{t:Mackey-space-group}) of Mackey groups among (some classes of) Mackey spaces.

The following fundamental question is posed in \cite{Gabr-A(s)-Mackey}: {\em Is it true that every lcs $E$ admits a Mackey group topology}?
It would be interesting to consider a more concrete problem in which $\mathcal{E}$ is a class of Mackey locally convex spaces from the first, second or forth lines of the diagram.
\begin{problem} \label{prob:lcs-E-Mackey-group}
Is it true that every $E\in \mathcal{E}$ admits a Mackey group topology?
\end{problem}


\bibliographystyle{amsplain}

\end{document}